\documentclass[11pt,reqno]{amsart}
\usepackage{amssymb,latexsym}
\usepackage{graphicx}
\usepackage{url}		%does nice formatting of URLs

\calclayout

\usepackage[utf8]{inputenc}
\usepackage[T1]{fontenc}
\usepackage{amsfonts}
\usepackage[greek,english]{babel}
\usepackage{lmodern}

\usepackage{enumerate}
\usepackage{enumitem}

\usepackage[normalem]{ulem} % for strike out: \sout

\makeatletter
\newcommand{\biggg}{\bBigg@{2.75}}
\newcommand{\Biggg}{\bBigg@{3}}
\newcommand{\bigggg}{\bBigg@{3.25}}
\newcommand{\Bigggg}{\bBigg@{3.5}}
\newcommand{\biggggg}{\bBigg@{3.75}}
\newcommand{\Biggggg}{\bBigg@{4}}
\makeatother

\newtheoremstyle{mystyle}%  % Name
  {}%                                     % Space above
  {}%                                     % Space below
  {\itshape}%                         % Body font
  {}%                                    % Indent amount
  {\bfseries}%                        % Theorem head font
  {.}%                                    % Punctuation after theorem head
  { }%                                    % Space after theorem head, ' ', or \newline
  {\thmname{#1}\thmnumber{ #2}\thmnote{ (#3)}} % Theorem head spec (can be left empty, meaning `normal')

\theoremstyle{mystyle}

\newtheorem{thm}{Theorem}[section]
\newtheorem{theorem}[thm]{Theorem}
\newtheorem{lemma}[thm]{Lemma}
\newtheorem{corollary}[thm]{Corollary}

\newtheorem{entry}[thm]{Entry}

\makeatletter
\newcommand{\leqnomode}{\tagsleft@true\let\veqno\@@leqno}
\newcommand{\reqnomode}{\tagsleft@false\let\veqno\@@eqno}
\makeatother

\allowdisplaybreaks
\begin{document}
\setcounter{page}{1}

\title[Ninth degree analogue]{Ninth degree analogue of \linebreak Ramanujan's septic theta function identity}
\author{Sun Kim and \"{O}rs Reb\'{a}k}
\address{Department of Mathematics, and Institute of Pure and Applied Mathematics, Jeonbuk National University, 567 Baekje-daero, Jeonju-si, Jeonbuk-do, 54896, Republic of Korea}
\email{sunkim@jbnu.ac.kr}
\address{Department of Mathematics and Statistics, University of Tromsø -- The Arctic University of Norway, 9037 Tromsø, Norway}
\email{ors.rebak@uit.no}
%\thanks{}

\dedicatory{Dedicated to George E. Andrews and Bruce C. Berndt for their 85th birthdays}

\begin{abstract}
On page 206 in his lost notebook, Ramanujan recorded a seventh degree identity for his theta function $\varphi(q)$. We give an analogous ninth degree identity. We also provide an application of an entry from his second notebook on a cubic equation and an interpretation with theta functions for some of his trigonometric identities. Lastly, we calculate five examples for $\varphi(e^{-\pi\sqrt{n}})$.
\end{abstract}

\keywords{theta functions, ninth degree identities, nonic identities, Ramanujan's lost notebook}
\subjclass[2010]{33C05, 05A30, 11F32, 11R29}
\date{\today}

\maketitle

\section{Introduction and the septic identity}

Ramanujan's general theta function $f(a,b)$ is defined by \cite[p.~197]{RamanujanEarlierII}, \cite[p.~34]{BerndtIII}
\begin{equation}\label{def:generaltheta}
f(a,b) := \sum_{n = -\infty}^\infty a^{n(n+1)/2}b^{n(n-1)/2}, \qquad |ab| < 1.
\end{equation}
The Jacobi triple product identity \cite[pp.~176--183]{Jacobi}, \cite[p.~197]{RamanujanEarlierII}, \cite[p.~35]{BerndtIII} states that
\begin{equation}\label{Jacobi-triple-product}
f(a,b) = (-a; ab)_\infty (-b; ab)_\infty (ab; ab)_\infty, \qquad  |ab| < 1,
\end{equation}
where
\begin{align*}
(a; q)_\infty := \prod_{k = 0}^\infty (1 - aq^k), \qquad |q| < 1.
\end{align*}
In Ramanujan's notation, the theta function $\varphi(q)$ is defined by \cite[p.~197]{RamanujanEarlierII}, \cite[p.~36]{BerndtIII}
\begin{equation}\label{def:phi}
\varphi(q) := f(q,q) = \sum_{n = -\infty}^\infty q^{n^2} = (-q; q^2)_\infty^2 (q^2; q^2)_\infty, \qquad |q| < 1,
\end{equation}
where the series and product representations follow from \eqref{def:generaltheta} and \eqref{Jacobi-triple-product}, respectively. After Ramanujan, we set \cite[p.~197]{RamanujanEarlierII}, \cite[p.~37]{BerndtIII}
\begin{equation}\label{def:chi}
\chi(q) := (-q; q^2)_\infty, \qquad |q| < 1.
\end{equation}
If $q = \exp(-\pi\sqrt{n})$, for a positive rational number $n$, the class invariant $G_n$ is defined by \cite{RamanujanModularPi}, \cite[pp.~23--39]{RamanujanCollected}, \cite[p.~183]{BerndtV}
\begin{equation}\label{def:G}
G_n := 2^{-1/4}q^{-1/24}\chi(q).
\end{equation}

Using the theory of modular equations, one can find values of $\varphi(q)$ of the form $\varphi(e^{-\pi \sqrt{n}})$, for a positive rational number $n$. For an overview of the theory of modular equations, see \cite[pp.~213--214]{BerndtIII}, \cite{BerndtChan}, \cite[p.~185]{BerndtV}, and the values obtained by Ramanujan are proved in \cite{BerndtChan}, \cite[pp.~323--351]{BerndtV}. We also recommend the survey \cite{BerndtRebak} in the subject. For calculating special values of $\varphi(q)$, there are other particular methods, as we shall see in Entry~\ref{entry:Ramanujan206}.

On page 206 in his lost notebook \cite{RamanujanLost}, Ramanujan gives the following septic identities.

\begin{entry}\label{entry:Ramanujan206}
\leqnomode
Let
\begin{equation}\label{7:1+u+v+w}
\frac{\varphi(q^{1/7})}{\varphi(q^7)} = 1 + u + v + w.\tag{i}
\end{equation}
Then
\begin{equation}\label{7:def:p}
p := uvw = \frac{8q^2 (-q; q^2)_\infty}{(-q^7; q^{14})^7_\infty}\tag{ii}
\end{equation}
and
\begin{align}\label{7:phi-to-p}
\frac{\varphi^8(q)}{\varphi^8(q^7)} - (2 + 5p)\frac{\varphi^4(q)}{\varphi^4(q^7)} + (1 - p)^3 = 0.\tag{iii}
\end{align}
Furthermore,
\begin{equation}\label{7:u,v,w}
u = \bigg(\frac{\alpha^2 p}{\beta}\bigg)^{1/7}, \quad
v = \bigg(\frac{\beta^2 p}{\gamma}\bigg)^{1/7}, \text{\quad and \quad}
w = \bigg(\frac{\gamma^2 p}{\alpha}\bigg)^{1/7},\tag{iv}
\end{equation}
where $\alpha, \beta,$ and $\gamma$ are the roots of the equation
\begin{equation}\label{7:r}
r(\xi) := \xi^3 + 2\xi^2\bigg(1 + 3p - \frac{\varphi^4(q)}{\varphi^4(q^7)}\bigg) + \xi p^2(p+4) - p^4 = 0.\tag{v}
\end{equation}
For example,
\begin{equation}\label{7:enigmatic}
\varphi(e^{-7\pi\sqrt{7}}) = 7^{-3/4}\varphi(e^{-\pi\sqrt{7}})\Big\{1 + (\quad)^{2/7} + (\quad)^{2/7} + (\quad)^{2/7}\Big\}.\tag{vi}
\end{equation}
\reqnomode
\end{entry}

We remark that $0 < |q| < 1$, and $u, v,$ and $w$ depend on the order of $\alpha, \beta,$ and $\gamma$. Parts~\eqref{7:1+u+v+w}--\eqref{7:r} were proved by Seung Hwan Son \cite{Son}, \cite[pp.~198--200]{Son2}, \cite[pp.~180--194]{AndrewsBerndtII}. Part~\eqref{7:enigmatic} was completed by the second author in \cite[Theorem~4.1]{Rebak}, by finding the three missing terms given as follows.

\begin{theorem}\label{thm:enigmatic} We have
\begin{equation*}
\varphi(e^{-7\pi\sqrt{7}}) = 7^{-3/4}\varphi(e^{-\pi\sqrt{7}})\Bigg\{1 + \bigg(\frac{\cos\frac{\pi}{7}}{2\cos^2\frac{2\pi}{7}}\bigg)^{2/7} + \bigg(\frac{\cos\frac{2\pi}{7}}{2\cos^2\frac{3\pi}{7}}\bigg)^{2/7} + \bigg(\frac{\cos\frac{3\pi}{7}}{2\cos^2\frac{\pi}{7}}\bigg)^{2/7}\Bigg\}.
\end{equation*}
\end{theorem}

Berndt and the second author \cite{BerndtRebak2} discussed cubic and quintic analogues of Entry~\ref{entry:Ramanujan206} and Theorem~\ref{thm:enigmatic}, with an extensive list of further examples. The aim of this present paper is to propose a ninth degree analogue.

In Section~\ref{section:ninth}, we give our ninth degree analogue of Entry~\ref{entry:Ramanujan206}. In Section~\ref{section:cubic}, we discuss an entry from Ramanujan's second notebook, and we offer an alternative formulation of our main result. In Section~\ref{section:examples}, we give examples, such as the analogous example for Theorem~\ref{thm:enigmatic}. 

For a rational number $n$, we evaluate values of $\varphi(e^{-\pi\sqrt{n}})$ in terms of trigonometric function values. The first four out of our five examples are given in \cite{BerndtRebak2} in an alternative formulation. We list them in the next theorem.

\begin{theorem}\label{thm:alternative} We have
\leqnomode
\begin{align}
\frac{\varphi(e^{-27\pi})}{\varphi(e^{-3\pi})} &= \frac{1}{3}\left(1 + (\sqrt{3} - 1)\left(\frac{(2(\sqrt{3} + 1))^{1/3} + 1}{(2(\sqrt{3} - 1))^{1/3} - 1}\right)^{1/3}\right)\tag{i}\label{eq:e27-alt},\\
\frac{\varphi(e^{-9\pi\sqrt{3}})}{\varphi(e^{-3\pi\sqrt{3}})} &= \frac{1}{3^{1/4}(1 + 2^{1/3})}\bigg(1 + \bigg(\frac{2}{2^{1/3} - 1}\bigg)^{1/3}\bigg)\label{eq:e9sqrt3-alt}\tag{ii},\\
\frac{\varphi(e^{-27\pi\sqrt{3}})}{\varphi(e^{-\pi\sqrt{3}})} &= \frac{1 + 2^{1/3}}{3^{7/4}}\Bigg(1 + 2^{1/3}(2^{1/3} - 1)\frac{2^{1/3} + 2^{2/3} + 3^{1/3}}{(9 - 2 \cdot 3^{4/3})^{1/3}}\Bigg)\label{eq:e27sqrt3-alt}\tag{iii},
\end{align}
and
\begin{multline}
\frac{\varphi(e^{-81\pi})}{\varphi(e^{-\pi})} = \frac{1 + (2(\sqrt{3} + 1))^{1/3}}{9}\\
\times\left(1 + 2^{1/3}\left(\frac{(2(\sqrt{3} - 1))^{1/3} - 1}{(2(\sqrt{3} + 1))^{1/3} + 1}\right)^{8/9}\left(\frac{3(\sqrt{3} + 1)}{\sqrt{3} + 1 - \left(\frac{(2(\sqrt{3} + 1))^{1/3} + 1}{(2(\sqrt{3} - 1))^{1/3} - 1}\right)^{1/3}} - 2\right)^{1/3}\right)\label{eq:e81-alt}\tag{iv}.
\end{multline}
\reqnomode
\end{theorem}

Theorem~\ref{thm:alternative}\eqref{eq:e27-alt} was first proved in print by Berndt and Chan in \cite{BerndtChan}. This value can be obtained by combining \cite[Corollaries~3.12 and~3.5]{BerndtRebak}. Likewise, to obtain Theorem~\ref{thm:alternative}\eqref{eq:e9sqrt3-alt}, combine \cite[Corollaries~3.7 and~3.4]{BerndtRebak}. Finally, Theorem~\ref{thm:alternative}\eqref{eq:e27sqrt3-alt},\eqref{eq:e81-alt} are given in \cite[Corollary~3.27]{BerndtRebak} and \cite[Corollary~3.28]{BerndtRebak}, respectively. All of these values can be expressed in terms of gamma functions. From \cite[pp. 524--525]{WhittakerWatson}, \cite[p.~103]{BerndtIII}, \cite{BerndtChan}, \cite[p.~325]{BerndtV}, and from \cite{Zucker}, \cite{BorweinZucker}, \cite[p.~298]{BorweinBrothersPiAGM},
\begin{equation*}
\varphi(e^{-\pi}) = \frac{\pi^{1/4}}{\Gamma\big(\frac{3}{4}\big)} = \frac{\Gamma\big(\frac{1}{4}\big)}{\sqrt{2} \pi^{3/4}} \qquad \text{and} \qquad
\varphi(e^{-\pi\sqrt{3}}) = \frac{3^{1/8}\Gamma^{3/2}\big(\frac{1}{3}\big)}{2^{2/3}\pi},
\end{equation*}
respectively. The values of $\varphi(e^{-3\pi})$ and $\varphi(e^{-3\pi\sqrt{3}})$ follow from \cite[Corollaries~3.5 and~3.4]{BerndtRebak}.

\section{A ninth degree analogue of Entry~\ref{entry:Ramanujan206}}\label{section:ninth}

In Theorem~\ref{thm:ninth}, we present a ninth degree analogue of Entry~\ref{entry:Ramanujan206}.

\begin{theorem}\label{thm:ninth} For $|q| < 1$, let
\leqnomode
\begin{equation}\label{def:u_k}
u_k := u_k(q) := \frac{2q^{k^2/9} f(q^{9 + 2k}, q^{9 - 2k})}{\varphi(q^9)}, \qquad k = 1, \dots, 4.\tag{0}
\end{equation}
Then,
\begin{equation}\label{9:phi-1-9}
\frac{\varphi(q^{1/9})}{\varphi(q^9)} = 1 + u_1 + u_2 + u_3 + u_4,\tag{i}
\end{equation}
\begin{equation}\label{9:p}
p := u_1 u_2 u_4 = \frac{8 q^{7/3} \chi(q)}{\chi^6(q^9)\chi(q^3)},\tag{ii}
\end{equation}
and
\begin{equation}\label{9:u3}
u_3 = \frac{\varphi(q)}{\varphi(q^9)} - 1.\tag{iii}
\end{equation}
Furthermore, for $0 < |q| < 1$,
\begin{equation}\label{9:u124}
u_1 = \bigg(\frac{\beta\, p}{\alpha}\bigg)^{1/3}, \quad
u_2 = \bigg(\frac{\gamma \, p}{\beta}\bigg)^{1/3}, \quad \text{and} \quad
u_4 = \bigg(\frac{\alpha\, p}{\gamma}\bigg)^{1/3},\tag{iv}
\end{equation}
where $\alpha, \beta,$ and $\gamma$ are the roots of the cubic equation
\begin{equation}\label{9:r}
r(\xi) := \xi^3 - u_3\bigg(\frac{\varphi^2(q)}{\varphi^2(q^9)} + 3\bigg)\bigg(\xi^2 - 2u_3^2\xi\bigg) - p^3 = 0.\tag{v}
\end{equation}
\reqnomode
\end{theorem}

Throughout the paper, we use the notation introduced in Theorem~\ref{thm:ninth}. Note that
\begin{equation*}
u_1 = \frac{2q^{1/9} f(q^{11}, q^7)}{\varphi(q^9)}, \quad 
u_2 = \frac{2q^{4/9} f(q^{13}, q^5)}{\varphi(q^9)}, \quad 
u_3 = \frac{2q f(q^{15}, q^3)}{\varphi(q^9)}, \quad 
u_4 = \frac{2q^{16/9} f(q^{17}, q)}{\varphi(q^9)}.
\end{equation*}
Also note that, by Theorem~\ref{thm:ninth}\eqref{9:u3}, we can express $r$ in terms of $u_3$ and $p$ as
\begin{equation*}
r(\xi) = \xi^3 - u_3(u_3^2 + 2u_3 + 4)\xi^2 + 2u_3^3(u_3^2 + 2u_3 + 4)\xi - p^3,
\end{equation*}
and as we shall see in Lemma~\ref{lemma:u_3}\eqref{u_3-iv}, $p^3$ is a rational function in $u_3$. 

There are some differences compared to Entry~\ref{entry:Ramanujan206}. Here, $u_3$ is not obtained from the roots of $r$, and in Theorem~\ref{thm:ninth}\eqref{9:u124}, the roots are employed in a different form than in Entry~\ref{entry:Ramanujan206}\eqref{7:u,v,w}. The order of the roots $\alpha, \beta,$ and $\gamma$ in Theorem~\ref{thm:ninth}\eqref{9:u124} is not provided. We shall give a solution for this in Lemma~\ref{lemma:9:a,b,c-order}. Without loss of generality, one can assume that
\begin{equation}\label{def:alpha,beta,gamma}
\alpha := u_2 u_4^2, \quad \beta := u_4 u_1^2, \quad \text{and} \quad \gamma := u_1 u_2^2.
\end{equation}

In order to establish Theorem~\ref{thm:ninth} toward the end of this section, we shall need some lemmas. We use Entry~31 of Chapter~16 of Ramanujan's second notebook \cite[p.~200]{RamanujanEarlierII}, \cite[pp.~48--49]{BerndtIII}.

\begin{lemma}\label{lemma:entry31} Let $\mathcal{U}_k = a^{k(k+1)/2}b^{k(k-1)/2}$ and $\mathcal{V}_k = a^{k(k-1)/2}b^{k(k+1)/2}$ for each nonnegative integer $k$. Then, for each positive integer $n$,
\begin{equation*}
f(\mathcal{U}_1, \mathcal{V}_1) =  \sum_{k=0}^{n-1} \mathcal{U}_k f\bigg(\frac{\mathcal{U}_{n+k}}{\mathcal{U}_k},\frac{\mathcal{V}_{n-k}}{\mathcal{U}_k}\bigg), \qquad |ab| < 1, \quad ab \neq 0.
\end{equation*}
\end{lemma}

\begin{proof}
See Berndt's book \cite[pp.~48--49, Entry~31]{BerndtIII}.
\end{proof}

Theorem~\ref{thm:ninth}\eqref{9:phi-1-9} holds in the following more general form.

\begin{lemma}\label{lemma:phi-sum-general} For any positive odd integer $n$ and $|q| < 1$,
\begin{equation*}
\frac{\varphi(q^{1/n})}{\varphi(q^n)} = 1 + \sum_{k = 1}^{(n - 1)/2} \frac{2 q^{k^2/n} f(q^{n + 2k}, q^{n - 2k})}{\varphi(q^n)}.
\end{equation*}
\end{lemma}

\begin{proof}
The case $q = 0$ is trivial. Otherwise, from Entry~18(i),(iv) in \cite[p.~197]{RamanujanEarlierII}, \cite[pp.~34--35]{BerndtIII}, we obtain the identity
\begin{equation}\label{f-a-1/a}
f(a, b) = a f(\tfrac{1}{a}, a^2 b), \qquad |ab| < 1, \quad a \neq 0.
\end{equation}
Now, for a positive integer $n$ and $0 < |q| < 1$, let 
\begin{equation}\label{def:u-tilde}
\widetilde{u}_k := q^{k^2/n} f(q^{n + 2k}, q^{n - 2k}), \qquad k = 0,\dots, n - 1.
\end{equation}
Because of \eqref{f-a-1/a}, we have $\widetilde{u}_k = \widetilde{u}_{n-k}$, for $k = 1, \dots, n - 1$. By applying Lemma~\ref{lemma:entry31} for $(a, b) = (q^{1/n}, q^{1/n})$, with the aid of \eqref{def:phi}, \eqref{f-a-1/a}, and \eqref{def:u-tilde}, for each odd $n$, we find that
\begin{equation*}
\varphi(q^{1/n}) = \sum_{k = 0}^{n - 1} \widetilde{u}_k = \varphi(q^n) + \sum_{k = 1}^{n - 1} \widetilde{u}_k = \varphi(q^n) + \sum_{k = 1}^{(n - 1)/2} 2\widetilde{u}_k.
\end{equation*}
After rearrangement, the proof is complete.
\end{proof}

Next, we propose some useful lemmas. Lemma~\ref{lemma:u_3} gives different representations for $u_3$, and a connection between $u_3(q^{1/3})$ and $u_3(q)$. Lemma~\ref{lemma:p} gives expressions for $p$.

\begin{lemma}\label{lemma:u_3} For $|q| < 1$,
\leqnomode
\begin{align}
u_3(q) &= \frac{\varphi(q)}{\varphi(q^9)} - 1\tag{i}\label{u_3-i},\\
u_3(q) &= \bigg(\frac{\varphi^4(q^3)}{\varphi^4(q^9)} - 1\bigg)^{1/3}\tag{ii}\label{u_3-ii},\\
u_3(q) &= \frac{2q\chi(q^3)}{\chi^3(q^9)}\tag{iii}\label{u_3-iii},
\intertext{and}
u_3(q^{1/3}) &= \Bigg(\frac{(u_3(q) + 1)^4}{u_3^3(q) + 1} - 1\Bigg)^{1/3}\tag{iv}\label{u_3-iv}.
\end{align}
\reqnomode
\end{lemma}

\begin{proof} Part~\eqref{u_3-i} follows from the first equation of~\cite[p.~49, Corollary(i)]{BerndtIII}. Then, \eqref{u_3-ii} follows from the first equation of~\cite[pp.~345--346, Entry~1(iii)]{BerndtIII}, by setting $q \mapsto q^3$. Now, \eqref{u_3-iii} is obtained by~\cite[p.~330, (4.6)]{BerndtV}, with $q \mapsto q^3$. From \eqref{u_3-i} and \eqref{u_3-ii}, we find
\begin{equation}\label{eq:u_3-eqs}
(u_3(q) + 1)^4 = \frac{\varphi^4(q)}{\varphi^4(q^9)} \quad \text{and} \quad u_3^3(q) + 1 = \frac{\varphi^4(q^3)}{\varphi^4(q^9)},
\end{equation}
respectively. Set $q \mapsto q^{1/3}$ in \eqref{u_3-ii} and combine it with \eqref{eq:u_3-eqs} to conclude \eqref{u_3-iv}.
\end{proof}

\begin{lemma}\label{lemma:p} For $|q| < 1$,
\leqnomode
\begin{align}
p &= \frac{8 q^{7/3} \chi(q)}{\chi^6(q^9)\chi(q^3)}\tag{i}\label{p-i},\\
p &= u_3^2(q) u_3(q^{1/3})\tag{ii}\label{p-ii},\\
p &= \bigg(\frac{\varphi(q)}{\varphi(q^9)} - 1\bigg)^2\bigg(\frac{\varphi^4(q)}{\varphi^4(q^3)} - 1\bigg)^{1/3}\tag{iii}\label{p-iii},
\end{align}
and
\begin{equation}
p^3(u_3^2 - u_3 + 1) = u_3^7(u_3^2 + 2u_3 + 4)\tag{iv}\label{p-iv}. 
\end{equation}
\reqnomode
\end{lemma}

\begin{proof} To prove \eqref{p-i}, we use an argument similar to that of Son~\cite[Lemma~4.1]{Son}, \cite[p.~183, Lemma~8.3.5]{AndrewsBerndtII}. The case $q = 0$ is trivial. Otherwise, apply the Jacobi triple product identity \eqref{Jacobi-triple-product} to establish the representations,
\begin{equation*}
f(q^{9 + 2k}, q^{9 - 2k}) = (-q^{9 + 2k}; q^{18})_\infty (-q^{9 - 2k}; q^{18})_\infty (q^{18}; q^{18})_\infty, \qquad k = 1,\dots,4,
\end{equation*}
and
\begin{equation*}
\varphi(q^9) = f(q^9, q^9) =  (-q^9; q^{18})_\infty^2 (q^{18}; q^{18})_\infty.
\end{equation*}
Thus,
\begin{equation*}
u_1 u_2 u_3 u_4 = \frac{16 q^{10/3}}{\varphi^4(q^9)}\prod_{k = 1}^4 f(q^{9 + 2k}, q^{9 - 2k}) = \frac{16 q^{10/3}}{(-q^9; q^{18})_\infty^8} \prod_{k = 1}^4 (-q^{9 + 2k}; q^{18})_\infty (-q^{9 - 2k}; q^{18})_\infty.
\end{equation*}
Since
\begin{equation*}
(-q; q^2)_\infty = \prod_{k = 1}^9 (-q^{2k - 1}; q^{18})_\infty,
\end{equation*}
and since $9 \pm 2k$, for $k = 1, \dots, 4$, covers $\{1, 3, 5, \dots, 17\} \setminus \{9\}$, and $2k - 1$, for $k = 1, \dots, 9$, covers $\{1, 3, 5, \dots, 17\}$, we find that
\begin{equation*}
u_1 u_2 u_3 u_4 = \frac{16 q^{10/3}(-q; q^2)_\infty}{(-q^9; q^{18})_\infty^9} = \frac{16 q^{10/3}\chi(q)}{\chi^9(q^9)},
\end{equation*}
where $\chi(q)$ is defined in \eqref{def:chi}. Using Lemma~\ref{lemma:u_3}\eqref{u_3-iii}, we conclude that
\begin{equation*}
p = u_1 u_2 u_4 = \frac{16 q^{10/3}\chi(q)}{\chi^9(q^9)} \cdot \frac{\chi^3(q^9)}{2q\chi(q^3)} = \frac{8 q^{7/3} \chi(q)}{\chi^6(q^9)\chi(q^3)}.
\end{equation*}

Now, to prove \eqref{p-ii}, by Lemma~\ref{lemma:u_3}\eqref{u_3-iii}, observe that
\begin{equation*}
p = \frac{8 q^{7/3} \chi(q)}{\chi^6(q^9)\chi(q^3)} = \frac{4 q^2 \chi^2(q^3)}{\chi^6(q^9)} \cdot \frac{2 q^{1/3} \chi(q)}{\chi^3(q^3)} = u_3^2(q) u_3(q^{1/3}).
\end{equation*}
Then, \eqref{p-iii} follows directly by Lemma~\ref{lemma:u_3}\eqref{u_3-i} and Lemma~\ref{lemma:u_3}\eqref{u_3-ii}. Lastly, \eqref{p-iv} can be deduced from \eqref{p-ii} and Lemma~\ref{lemma:u_3}\eqref{u_3-iv}, as
\begin{equation*}
p^3 =  u_3^6(q) u_3^3(q^{1/3}) = u_3^6 \cdot \Bigg(\frac{(u_3 + 1)^4}{u_3^3 + 1} - 1\Bigg) = \frac{u_3^7(u_3^2 + 2u_3 + 4)}{u_3^2 - u_3 + 1}.
\end{equation*}
After rearrangement, we obtain the desired result.
\end{proof}

In the following lemma, we express $u_1 + u_2 + u_4$ in two ways, first with the aid of $\varphi^4(q^{1/3})$, and second in terms of $u_3$ and $p$.

\begin{lemma}\label{lemma:u_1+u_2+u_4} For $|q| < 1$,
\leqnomode
\begin{align}
(u_1 + u_2 + u_4)^3 &= \frac{\varphi^3(q)}{\varphi^3(q^9)}\bigg(\frac{\varphi^4(q^{1/3})}{\varphi^4(q)} - 1\bigg)\tag{i}\label{u_1+u_2+u_4-i},\\
\intertext{and for $0 < |q| < 1$,}
(u_1 + u_2 + u_4)^3 &= (u_3 + 1)^3\left(\frac{\left(\frac{p}{u_3^2} + 1\right)^4}{\left(\frac{p}{u_3^2}\right)^3 + 1} - 1\right)\tag{ii}\label{u_1+u_2+u_4-ii}.
\end{align}
\reqnomode
\end{lemma}

\begin{proof} By Lemma~\ref{lemma:phi-sum-general} for $n = 9$ and Lemma~\ref{lemma:u_3}\eqref{u_3-i}, or from \cite[pp.~349--352, Entry~2(ix)]{BerndtIII},
\begin{equation*}
\frac{\varphi(q^{1/9})}{\varphi(q^9)} = \frac{\varphi(q)}{\varphi(q^9)} + u_1 + u_2 + u_4. 
\end{equation*}
After rearrangement, by Lemma~\ref{lemma:u_3}\eqref{u_3-i},\eqref{u_3-ii}, we deduce that
\begin{equation}\label{eq:u_1+u_2+u_4-eqs-1}
 u_1 + u_2 + u_4 = \frac{\varphi(q)}{\varphi(q^9)}\bigg(\frac{\varphi(q^{1/9})}{\varphi(q)} - 1\bigg)= \frac{\varphi(q)}{\varphi(q^9)}u_3(q^{1/9}) = \frac{\varphi(q)}{\varphi(q^9)}\bigg(\frac{\varphi^4(q^{1/3})}{\varphi^4(q)} - 1\bigg)^{1/3},
\end{equation}
and by Lemma~\ref{lemma:u_3}\eqref{u_3-i},\eqref{u_3-iv}, and Lemma~\ref{lemma:p}\eqref{p-ii}, we find that
\begin{multline}\label{eq:u_1+u_2+u_4-eqs-2}
 u_1 + u_2 + u_4 = \frac{\varphi(q)}{\varphi(q^9)}\bigg(\frac{\varphi(q^{1/9})}{\varphi(q)} - 1\bigg) = (u_3(q) + 1)u_3(q^{1/9})\\
 = (u_3(q) + 1)\Bigg(\frac{(u_3(q^{1/3}) + 1)^4}{u_3^3(q^{1/3}) + 1} - 1\Bigg)^{1/3} = (u_3 + 1)\left(\frac{\left(\frac{p}{u_3^2} + 1\right)^4}{\left(\frac{p}{u_3^2}\right)^3 + 1} - 1\right)^{1/3}.
\end{multline}
Cubing both sides of \eqref{eq:u_1+u_2+u_4-eqs-1} and \eqref{eq:u_1+u_2+u_4-eqs-2}, we obtain \eqref{u_1+u_2+u_4-i} and \eqref{u_1+u_2+u_4-ii}, respectively.
\end{proof}

\begin{lemma}\label{lemma:u_3-and-p} For $0 < |q| < 1$,
\begin{equation*}
\frac{p}{u_3^2}(u_3^3 + 4) + 6p + 3u_3(u_3^2 + 2u_3 + 4)\bigg(\frac{2 u_3^2}{p} + 1\bigg) = (u_3 + 1)^3\left(\frac{\left(\frac{p}{u_3^2} + 1\right)^4}{\left(\frac{p}{u_3^2}\right)^3 + 1} - 1\right).
\end{equation*}
\end{lemma}

\begin{proof}
After rearrangement, by Lemma~\ref{lemma:p}\eqref{p-iv}, we find that
\begin{multline*}
\frac{p}{u_3^2}(u_3^3 + 4) + 6p + 3u_3(u_3^2 + 2u_3 + 4)\bigg(\frac{2 u_3^2}{p} + 1\bigg) - (u_3 + 1)^3\left(\frac{\left(\frac{p}{u_3^2} + 1\right)^4}{\left(\frac{p}{u_3^2}\right)^3 + 1} - 1\right)\\
= \frac{3(p - 2u_3^2)}{p u_3^2 (p^2 - p u_3^2 + u_3^4)}(p^3(u_3^2 - u_3 + 1) - u_3^7(u_3^2 + 2u_3 + 4)) = 0.\tag*{\qedhere}
\end{multline*}
\end{proof}

For now, let $\{x\} \in [0,1)$ denote the fractional part of a real number $x$. Inspired by Son's argument \cite{Son}, \cite[pp.~180--194]{AndrewsBerndtII}, for a real number $\alpha \in [0,1)$, an enumerable set of real numbers $S$, and a power series $\sum_{v \in S} a_v q^v$, we define the operator $\mathcal{M_\alpha}$ as
\begin{equation}\label{def:M}
\mathcal{M_\alpha}\Bigg(\sum_{v \in S} a_v q^v\Bigg) := \sum_{\substack{w \in S\\ \{w\} = \alpha}} a_w q^w.
\end{equation}

Note that $\mathcal{M}_\alpha$ collects the terms of a power series whose exponents have fractional part $\alpha$. Also note that since power series representations are unique, by applying $\mathcal{M}_\alpha$ on an equation, we obtain an identity. A more general concept is introduced in \cite[Definition~2.8]{Son3}.

\begin{lemma}\label{lemma:u_1,u_2,u_4-split} For $|q| < 1$,
\leqnomode
\begin{align}
u_1 u_2^2 + u_2 u_4^2 + u_4 u_1^2 &= u_3\bigg(\frac{\varphi^2(q)}{\varphi^2(q^9)} + 3\bigg)\tag{i}\label{split-i},\\
p(u_1^2 u_2 + u_2^2 u_4 + u_4^2 u_1) &= 2 u_3^3 \bigg(\frac{\varphi^2(q)}{\varphi^2(q^9)} + 3\bigg)\tag{ii}\label{split-ii},\\
\intertext{and for $0 < |q| < 1$,}
u_1^3 + u_2^3 + u_4^3 &= \frac{p}{u_3^2}(u_3^3 + 4)\tag{iii}\label{split-iii}.
\end{align}
\reqnomode
\end{lemma}

\begin{proof} The case $q = 0$ is trivial. Otherwise, combining Lemma~\ref{lemma:u_3-and-p} and Lemma~\ref{lemma:u_1+u_2+u_4}\eqref{u_1+u_2+u_4-ii}, we have
\begin{equation}
\frac{p}{u_3^2}(u_3^3 + 4) + 6p + 3u_3(u_3^2 + 2u_3 + 4)\bigg(\frac{2 u_3^2}{p} + 1\bigg) = (u_1 + u_2 + u_4)^3.\label{eq:split-eq}
\end{equation}
Expanding the right-hand side of \eqref{eq:split-eq} and using Lemma~\ref{lemma:u_3}\eqref{u_3-i}, after simplification, we obtain
\begin{multline}
\frac{p}{u_3^2}(u_3^3 + 4) + 3u_3\bigg(\frac{\varphi^2(q)}{\varphi^2(q^9)} + 3\bigg)\bigg(\frac{2 u_3^2}{p} + 1\bigg)\\
= u_1^3 + u_2^3 + u_4^3 + 3(u_1 u_2^2 + u_2 u_4^2 + u_4 u_1^2 + u_1^2 u_2 + u_2^2 u_4 + u_4^2 u_1).\label{eq:split-eq-2}
\end{multline}

Now, using the series representations of $u_1, u_2, u_3,$ and $u_4$, and applying $\mathcal{M}_0$, defined in \eqref{def:M}, on both sides of \eqref{eq:split-eq-2}, i.e., collecting the terms with integer powers of $q$, we obtain part~\eqref{split-i}. Similarly, applying $\mathcal{M}_{2/3}$ and  $\mathcal{M}_{1/3}$ on \eqref{eq:split-eq-2}, we arrive at \eqref{split-ii} and \eqref{split-iii}, respectively. 
\end{proof}

Now, we summarize our proof of Theorem~\ref{thm:ninth} as the following.

\begin{proof}[Proof of Theorem~\ref{thm:ninth}] Part~\eqref{9:phi-1-9} is obtained by Lemma~\ref{lemma:phi-sum-general} with $n = 9$. Then, \eqref{9:p} and \eqref{9:u3} are given in Lemma~\ref{lemma:p}\eqref{p-i} and Lemma~\ref{lemma:u_3}\eqref{u_3-i}, respectively. Assuming~\eqref{def:alpha,beta,gamma}, the construction in \eqref{9:u124} holds, since the coefficients of the quadratic, linear, and constant terms of $r$ in \eqref{9:r} follow by Lemma~\ref{lemma:u_1,u_2,u_4-split}\eqref{split-i},\eqref{split-ii} and \eqref{def:alpha,beta,gamma}, respectively. These are
\begin{align*}
\alpha + \beta + \gamma &= u_1 u_2^2 + u_2 u_4^2 + u_4 u_1^2 = u_3\bigg(\frac{\varphi^2(q)}{\varphi^2(q^9)} + 3\bigg),\\
\alpha\beta + \beta\gamma + \gamma\alpha &= p(u_1^2 u_2 + u_2^2 u_4 + u_4^2 u_1) = 2 u_3^3 \bigg(\frac{\varphi^2(q)}{\varphi^2(q^9)} + 3\bigg),\\
\alpha\beta\gamma &= (u_1 u_2 u_4)^3 = p^3.
\end{align*}
Thus, $\alpha, \beta,$ and $\gamma$ are the roots of $\eqref{9:r}$.
\end{proof}

\begin{corollary}\label{cor:u_1,u_2,u_4-reciprocal} For $0 < |q| < 1$,
\begin{equation*}
\frac{1}{u_1^3} + \frac{1}{u_2^3} + \frac{1}{u_4^3} = \frac{1}{p u_3^3}(u_3^4 - u_3^3 + 6u_3^2 - 8u_3 + 8).
\end{equation*}
\end{corollary}

\begin{proof} First, by multiplying the left-hand sides of Lemma~\ref{lemma:u_1,u_2,u_4-split}\eqref{split-i} and Lemma~\ref{lemma:u_1,u_2,u_4-split}\eqref{split-ii}, after rearrangement, we find that
\begin{equation*}
\frac{1}{u_1^3} + \frac{1}{u_2^3} + \frac{1}{u_4^3} =
\frac{1}{p^3}(u_1 u_2^2 + u_2 u_4^2 + u_4 u_1^2)(u_1^2 u_2 + u_2^2 u_4 + u_4^2 u_1) - \frac{1}{p^2}(u_1^3 + u_2^3 + u_4^3) - \frac{3}{p}.
\end{equation*}
Then, by Lemma~\ref{lemma:u_1,u_2,u_4-split}\eqref{split-i},\eqref{split-ii},\eqref{split-iii}, Lemma~\ref{lemma:u_3}\eqref{u_3-i}, and Lemma~\ref{lemma:p}\eqref{p-iv}, we deduce that
\begin{align*}
\frac{1}{u_1^3} + \frac{1}{u_2^3} + \frac{1}{u_4^3} 
&= \frac{2u_3^4}{p^4}(u_3^2 + 2u_3 + 4)^2 - \frac{1}{u_3^2 p}(u_3^3 + 4) - \frac{3}{p}\\
&= \frac{1}{p u_3^3} \bigg(\frac{2 u_3^7}{p^3}(u_3^2 + 2u_3 + 4)^2 - u_3(u_3^3 + 4) - 3u_3^3\bigg)\\
&= \frac{1}{p u_3^3} \bigg(2\frac{u_3^2 - u_3 + 1}{u_3^2 + 2u_3 + 4}(u_3^2 + 2u_3 + 4)^2 - u_3(u_3^3 + 4) - 3u_3^3\bigg)\\
&= \frac{1}{p u_3^3}(u_3^4 - u_3^3 + 6u_3^2 - 8u_3 + 8).\tag*{\qedhere}
\end{align*}
\end{proof}

\section{Ramanujan's entry on a cubic equation}\label{section:cubic}

On page 356 in Ramanujan's second notebook~\cite{RamanujanEarlierII}, he states the following algebraic identity.

\begin{entry}\label{entry:Ramanujan356} If $\alpha, \beta,$ and $\gamma$ are the roots of
\begin{equation*}
x^3 - ax^2 + bx - 1 = 0,
\end{equation*}
then
\begin{equation}\label{eq:entry30-roots}
(\alpha/\beta)^{1/3} + (\beta/\gamma)^{1/3} + (\gamma/\alpha)^{1/3} \text{\quad and \quad}
(\beta/\alpha)^{1/3} + (\gamma/\beta)^{1/3} + (\alpha/\gamma)^{1/3}
\end{equation}
are the roots of
\begin{equation*}
z^2 - tz + a + b + 3 = 0,
\end{equation*}
where
\begin{equation*}
t^3 - 3(a + b +3)t - (ab + 6(a + b) + 9) = 0.
\end{equation*}
The quantities \eqref{eq:entry30-roots} are also roots of
\begin{equation}\label{eq:entry30-6th}
y^6 - (ab + 6(a + b) + 9)y^3 + (a + b + 3)^3 = 0.
\end{equation}
Lastly,
\begin{multline*}
\alpha^{1/3} + \beta^{1/3} + \gamma^{1/3}\\
= \bigg(a + 6 + 3\bigg\{\frac{ab + 9}{2} + 3(a + b) + \bigg(\bigg(\frac{ab + 9}{2}\bigg)^2 - a^3 - b^3 - 27\bigg)^{1/2}\bigg\}^{1/3}\\
+ 3\bigg\{\frac{ab + 9}{2} + 3(a + b) - \bigg(\bigg(\frac{ab + 9}{2}\bigg)^2 - a^3 - b^3 - 27\bigg)^{1/2}\bigg\}^{1/3}\bigg)^{1/3}.
\end{multline*}
\end{entry}

Entry~\ref{entry:Ramanujan356} is a sequel to Ramanujan's earlier result \cite[pp.~22--23, Entry~10]{BerndtIV}, \cite[pp.~652--653]{BerndtBhargava}.

\begin{proof} See Berndt's book~\cite[pp.~40--41, Entry~30]{BerndtIV}.
\end{proof}

For our purposes, we restate some parts of Entry~\ref{entry:Ramanujan356} as the following.

\begin{lemma}\label{lemma:Ramanujan356restate} If $\alpha, \beta,$ and $\gamma$ are the roots of
\begin{equation*}
x^3 - a x^2 + b x - 1 = 0,
\end{equation*}
and if $\mu$ and $\nu$ are the roots of
\begin{equation*}
y^2 - (ab + 6(a + b) + 9)y + (a + b + 3)^3 = 0,
\end{equation*}
then
\begin{equation*}
\mu^{1/3} = (\alpha/\beta)^{1/3} + (\beta/\gamma)^{1/3} + (\gamma/\alpha)^{1/3} \quad \text{and} \quad
\nu^{1/3} = (\beta/\alpha)^{1/3} + (\gamma/\beta)^{1/3} + (\alpha/\gamma)^{1/3} .
\end{equation*}
\end{lemma}

\begin{proof}
We utilize \eqref{eq:entry30-roots} and \eqref{eq:entry30-6th} of Entry~\ref{entry:Ramanujan356}. Replace $y$ by $y^{1/3}$ in \eqref{eq:entry30-6th} to obtain the result.
\end{proof}

Note the symmetry in Lemma~\ref{lemma:Ramanujan356restate}. Interchanging $a$ and $b$ produces the same case.

\begin{theorem}\label{thm:ninthrestate} Let $u_3$ and $p$ be as defined in Theorem~\ref{thm:ninth}. Furthermore, for $0 < |q| < 1$, let 
\begin{equation*}
a := \frac{u_3}{p}\bigg(\frac{\varphi^2(q)}{\varphi^2(q^9)} + 3\bigg) \quad \text{and} \quad 
b := \frac{2 u_3^3}{p^2}\bigg(\frac{\varphi^2(q)}{\varphi^2(q^9)} + 3\bigg) = \frac{2 u_3^2 a}{p}.
\end{equation*}
Then,
\begin{equation*}
\frac{\varphi(q^{1/9})}{\varphi(q^9)} = 1 + u_3 + (py)^{1/3},
\end{equation*}
where
\begin{equation*}
y^2 - (ab + 6(a + b) + 9)y + (a + b + 3)^3 = 0.
\end{equation*}
\end{theorem}

Theorem~\ref{thm:ninthrestate} is a corollary of Theorem~\ref{thm:ninth} and Lemma~\ref{lemma:Ramanujan356restate}.

\begin{proof} From Theorem~\ref{thm:ninth}\eqref{9:r}, consider the polynomial
\begin{equation}\label{eq:restate-r}
r(\xi) = \xi^3 - a p \xi^2 + b p^2 \xi - p^3 = 0.
\end{equation}
By using an appropriate polynomial transformation, we find that $\alpha, \beta,$ and $\gamma$ are the roots of \eqref{eq:restate-r} exactly if $\alpha/p,\, \beta/p,$ and $\gamma/p$ are the roots of
\begin{equation}\label{eq:tranform-r}
\frac{r(x p)}{p^3} = x^3 - a x^2 + b x - 1 = 0.
\end{equation}
The rest of the proof follows directly by Theorem~\ref{thm:ninth}\eqref{9:phi-1-9},\eqref{9:u124} and Lemma~\ref{lemma:Ramanujan356restate}.
\end{proof}

\section{Examples for ninth degree identities}\label{section:examples}

In this section, we present five ninth degree examples. To obtain examples for Theorem~\ref{thm:ninth}, we need the values of triplets of class invariants $G_n, G_{9n},$ and $G_{81n}$, for certain positive rational numbers $n$. Ramanujan~\cite[pp.~189--199]{BerndtV} calculated the class invariant $G_n$ for a total of $78$ values of $n$. In our first four examples, we use the values for $n = 1, 3, 9, 27,$ and $81$. In our fifth example, we use $G_{243}$, given by Watson~\cite{Watson3}. We give our values of Theorem~\ref{thm:ninth}\eqref{9:phi-1-9} in trigonometric form. As we mentioned in the introduction, alternative expressions for Theorems~\ref{thm:e27}--\ref{thm:e81} are given in terms of radicals in Theorem~\ref{thm:alternative}\eqref{eq:e27-alt}--\eqref{eq:e81-alt}. Theorem~\ref{thm:e81sqrt3} appears to be new. Our examples are summarized in Table~\ref{table:ninth}.

\begin{table}[ht]
\centering
\begin{tabular}{| r | r | r | l | l | l |}
	\hline &&&&&\\[-1em]
	$n$ & $9n$ & $81n$ & Value & Ex. for Thm.~\ref{thm:ninth} & Other representation\\
	\hline &&&&&\\[-1em]
	$1/9$ & $1$ & $9$ & $\varphi(e^{-27\pi})/\varphi(e^{-3\pi})$ & Theorem~\ref{thm:e27} & Theorem~\ref{thm:alternative}\eqref{eq:e27-alt} \\
	$1/3$ & $3$ & $27$ & $\varphi(e^{-9\pi\sqrt{3}})/\varphi(e^{-3\pi\sqrt{3}})$ & Theorem~\ref{thm:e9sqrt3} & Theorem~\ref{thm:alternative}\eqref{eq:e9sqrt3-alt} \\
	$1/27$ & $1/3$ & $3$ & $\varphi(e^{-27\pi\sqrt{3}})/\varphi(e^{-\pi\sqrt{3}})$ & Theorem~\ref{thm:e27sqrt3} & Theorem~\ref{thm:alternative}\eqref{eq:e27sqrt3-alt} \\
	$1/81$ & $1/9$ & $1$ & $\varphi(e^{-81\pi})/\varphi(e^{-\pi})$ & Theorem~\ref{thm:e81} & Theorem~\ref{thm:alternative}\eqref{eq:e81-alt} \\
	\hline &&&&&\\[-1em]
	$1/243$ & $1/27$ & $1/3$ & $\varphi(e^{-81\pi\sqrt{3}})/\varphi(e^{-\pi\sqrt{3}})$ & Theorem~\ref{thm:e81sqrt3} & \multicolumn{1}{|c|}{---} \\
	\hline
\end{tabular}
\vspace*{1em}
\caption{Overview of ninth degree examples}\label{table:ninth}
\end{table}

\vspace{-1em}

First, we shall need some lemmas. Lemmas~\ref{lemma:transform}--\ref{lemma:G81n} are classical. Lemmas~\ref{lemma:p-u_3-G},~\ref{lemma:u_1,u_2,u_3,u_4,p-supp}, and \ref{lemma:9:a,b,c-order} are analogous to the septic case lemmas in \cite[Lemmas~2.3,~3.3, and 3.7]{Rebak}, respectively.

\begin{lemma}\label{lemma:transform} If $n$ is a positive rational number, then
\begin{equation*}
\varphi(e^{-\pi/\sqrt{n}}) = n^{1/4} \varphi(e^{-\pi\sqrt{n}}).
\end{equation*}
\end{lemma}

\begin{proof} The transformation formula for $\varphi(q)$ states that if $\operatorname{Re}(\alpha^2), \operatorname{Re}(\beta^2) > 0$ and $\alpha\beta = \pi$, then \cite[p.~199]{RamanujanEarlierII}, \cite[p.~43]{BerndtIII}
\begin{equation*}
\sqrt{\alpha} \varphi\big(e^{-\alpha^2}\big) = \sqrt{\beta} \varphi\big(e^{-\beta^2}\big).
\end{equation*}
The lemma is the special case for $\alpha^2 = \pi/\sqrt{n}$.
\end{proof}

\begin{lemma}\label{lemma:G} If $n$ is a positive rational number, then $G_n = G_{1/n}$.
\end{lemma}

\begin{proof}
See Ramanujan's paper \cite{RamanujanModularPi}, \cite[pp.~23--39]{RamanujanCollected} or Yi's thesis \cite[pp.~18--19]{YiThesis}.
\end{proof}

\begin{lemma}\label{lemma:9n} If $n$ is a positive rational number, then
\begin{equation*}
\frac{\varphi(e^{-9\pi\sqrt{n}})}{\varphi(e^{-\pi\sqrt{n}})} = \frac{1}{3}\bigg(1 + \frac{\sqrt{2} G_{9n}}{G_n^3}\bigg).
\end{equation*}
\end{lemma}

\begin{proof}
See \cite[p.~145]{BorweinBrothersPiAGM}, \cite[p.~334, (5.7)]{BerndtV}, \cite[(3.30)]{BerndtChan}, or \cite[Theorem~3.3]{BerndtRebak2}.
\end{proof}

Berndt~\cite[pp.~205--207, Theorem~3.1]{BerndtV} expresses $G_{9n}$ in terms of $G_n$. The Borwein brothers \cite[p.~145, (4.7.12)]{BorweinBrothersPiAGM} connect $G_{81n}$ with $G_{9n}$ and $G_n$, which we provide in our next lemma.

\begin{lemma}\label{lemma:G81n} If $n$ is a positive rational number, then
\begin{equation*}
G_{81n}^3 = G_{9n} \frac{\sqrt{2} G_{9n} + G_n^3}{\sqrt{2} G_n^3 - G_{9n}}.
\end{equation*}
\end{lemma}

\begin{proof}
See the Borweins' book \cite[p.~145, (4.7.12)]{BorweinBrothersPiAGM}.
\end{proof}

To find their values, in our next lemma, we express $p$ and $u_3$ in terms of $G_n$.

\begin{lemma}\label{lemma:p-u_3-G} If $q = \exp(-\pi\sqrt{n})$, for a positive rational number $n$, then
\leqnomode
\begin{align}
p &= \frac{2\sqrt{2} G_n}{G_{9n} G_{81n}^6}\tag{i}\label{p-G}\\
\intertext{and}
u_3 &= \frac{\sqrt{2} G_{9n}}{G_{81n}^3}\tag{ii}\label{u_3-G}.
\end{align}
\reqnomode
\end{lemma}

\begin{proof} To prove \eqref{p-G}, from Lemma~\ref{lemma:p}\eqref{p-i} and \eqref{def:G}, we have
\begin{equation*}
p = \frac{8 q^{7/3} \chi(q)}{\chi^6(q^9)\chi(q^3)} = \frac{2 \sqrt{2} G_n}{G_{81n}^6 G_{9n}}.
\end{equation*}
Similarly, to obtain \eqref{u_3-G}, from Lemma~\ref{lemma:u_3}\eqref{u_3-iii} and \eqref{def:G}, we deduce
\begin{equation*}
u_3 = \frac{2q\chi(q^3)}{\chi^3(q^9)} = \frac{\sqrt{2} G_{9n}}{G_{81n}^3}.\tag*{\qedhere}
\end{equation*}
\end{proof}

Next, we prove a monotonicity property of the values $u_k$ defined in Theorem~\ref{thm:ninth}\eqref{def:u_k}.

\begin{lemma}\label{lemma:u_1,u_2,u_3,u_4,p-supp} For $0 < q < 1$, 
\begin{equation*}
2 > u_1 > u_2 > u_3 > u_4 > 0.
\end{equation*}
\end{lemma}

\begin{proof}
From \cite[Lemma~3.2]{Rebak}, if $n$ is a nonnegative integer and $0 < q < 1$, then $\widetilde{u}_k$ defined in \eqref{def:u-tilde} is positive and strictly monotonically decreasing for $k = 0, \dots, n/2$, when $n$ is even, and for $k = 0, \dots, (n-1)/2$, when $n$ is odd. The lemma follows from this result with $n = 9$.
\end{proof}

We shall need the following inequality.

\begin{lemma}\label{lemma:u_1,u_2,u_4-inequality} For $0 < q < 1$,
\begin{equation*}
\frac{p}{u_1^3} + \frac{p}{u_2^3} + \frac{p}{u_4^3} > \frac{u_1^3}{p} + \frac{u_2^3}{p} + \frac{u_4^3}{p}.
\end{equation*}
\end{lemma}

\begin{proof}
By Corollary~\ref{cor:u_1,u_2,u_4-reciprocal} and Lemma~\ref{lemma:u_1,u_2,u_4-split}\eqref{split-iii}, we have
\begin{equation*}
\bigg(\frac{p}{u_1^3} + \frac{p}{u_2^3} + \frac{p}{u_4^3}\bigg) - \bigg(\frac{u_1^3}{p} + \frac{u_2^3}{p} + \frac{u_4^3}{p}\bigg) 
= \frac{u_3^4 - u_3^3 + 6u_3^2 - 8u_3 + 8}{u_3^3} - \frac{u_3^3 + 4}{u_3^2} = \bigg(\frac{2 - u_3}{u_3}\bigg)^3 > 0,
\end{equation*}
where the inequality holds by Lemma~\ref{lemma:u_1,u_2,u_3,u_4,p-supp}.
\end{proof}

The next lemma provides the correct order of the roots of $r$ in the case of $0 < q < 1$.

\begin{lemma}\label{lemma:9:a,b,c-order} For $0 < q < 1$, suppose that the roots of $r$ are given in order $(\alpha, \beta, \gamma)$ such that they satisfy the following two conditions:
\leqnomode
\begin{equation}\label{9:a,b,c-order-cond1}
(\alpha - \beta)(\beta - \gamma)(\gamma - \alpha) > 0\tag{i}
\end{equation}
and
\begin{equation}\label{9:a,b,c-order-cond2}
\frac{\beta}{\alpha} > \frac{\gamma}{\beta} > \frac{\alpha}{\gamma}.\tag{ii}
\end{equation}
\reqnomode
Then
\begin{equation*}
u_1 = \bigg(\frac{\beta\, p}{\alpha}\bigg)^{1/3}, \quad
u_2 = \bigg(\frac{\gamma\, p}{\beta}\bigg)^{1/3}, \text{\quad and \quad}
u_4 = \bigg(\frac{\alpha\, p}{\gamma}\bigg)^{1/3}.
\end{equation*}
\end{lemma}

Condition~\eqref{9:a,b,c-order-cond1} guarantees the correct order of $\alpha, \beta$, and $\gamma$ in Theorem~\ref{thm:ninth}\eqref{9:u124}, so that Theorem~\ref{thm:ninth}\eqref{9:phi-1-9} holds. Condition~\eqref{9:a,b,c-order-cond2} is needed to fulfill Theorem~\ref{thm:ninth}\eqref{def:u_k}.

\begin{proof}
For each possible order of $\alpha, \beta,$ and $\gamma$, consider the set of possible values of $u_1, u_2,$ and $u_4$ given in Theorem~\ref{thm:ninth}\eqref{9:u124}. For $(\alpha, \beta, \gamma), (\beta, \gamma, \alpha), (\gamma, \alpha, \beta)$ and for $(\gamma, \beta, \alpha), (\beta, \alpha, \gamma), (\alpha, \gamma, \beta)$, we have
\begin{equation}\label{eq:9:a,b,c-possible-sets}
\Bigg\{\bigg(\frac{\beta\, p}{\alpha}\bigg)^{1/3},\, \bigg(\frac{\gamma\, p}{\beta}\bigg)^{1/3},\, \bigg(\frac{\alpha\, p}{\gamma}\bigg)^{1/3}\Bigg\}\text{\quad and \quad} \Bigg\{\bigg(\frac{\alpha\, p}{\beta}\bigg)^{1/3},\,\bigg(\frac{\beta\, p}{\gamma}\bigg)^{1/3},\, \bigg(\frac{\gamma\, p}{\alpha}\bigg)^{1/3}\Bigg\}\,,
\end{equation}
respectively. Thus, it is enough to consider the order $(\alpha, \beta, \gamma)$ and its reverse $(\gamma, \beta, \alpha)$. By Lemma~\ref{lemma:u_1,u_2,u_3,u_4,p-supp}, $p$ is positive. By condition~\eqref{9:a,b,c-order-cond1}, since $\alpha\beta\gamma = p^3 > 0$, we find that
\begin{equation*}
\frac{(\alpha - \beta)(\beta - \gamma)(\gamma - \alpha)}{\alpha\beta\gamma} > 0,
\end{equation*}
and therefore, after rearrangement,
\begin{equation}\label{eq:9:a,b,c-order-cond1-alt}
\frac{\alpha}{\beta} + \frac{\beta}{\gamma} + \frac{\gamma}{\alpha} > 
\frac{\beta}{\alpha} + \frac{\gamma}{\beta} + \frac{\alpha}{\gamma}.
\end{equation}
In view of Theorem~\ref{thm:ninth}\eqref{9:u124} and Lemma~\ref{lemma:u_1,u_2,u_4-inequality}, by Lemma~\ref{lemma:u_1,u_2,u_4-split}\eqref{split-iii} and Corollary~\ref{cor:u_1,u_2,u_4-reciprocal}, we obtain
\begin{equation}\label{eq:9:a,b,c-order-alt-cond1}
\frac{\beta}{\alpha} + \frac{\gamma}{\beta} + \frac{\alpha}{\gamma} = \frac{u_3^3 + 4}{u_3^2} \quad \text{and} \quad \frac{\alpha}{\beta} + \frac{\beta}{\gamma} + \frac{\gamma}{\alpha} = \frac{u_3^4 - u_3^3 + 6u_3^2 - 8u_3 + 8}{u_3^3},
\end{equation}
respectively. Thus, we find that the first set in \eqref{eq:9:a,b,c-possible-sets} is the correct choice.

Since $p > 0$, the condition \eqref{9:a,b,c-order-cond2} guarantees that $u_1 > u_2 > u_4$, which holds by Lemma~\ref{lemma:u_1,u_2,u_3,u_4,p-supp}.
\end{proof}

Because of Corollary~\ref{cor:u_1,u_2,u_4-reciprocal}, the condition in Lemma~\ref{lemma:9:a,b,c-order}\eqref{9:a,b,c-order-cond1} is simpler than in \cite[Lemma~3.7(i)]{Rebak}, but one could replace it with one of the equations in \eqref{eq:9:a,b,c-order-alt-cond1}. In some cases, it might be useful to rewrite Lemma~\ref{lemma:9:a,b,c-order}\eqref{9:a,b,c-order-cond1} into the form of \eqref{eq:9:a,b,c-order-cond1-alt}. As in \cite[Lemma~3.6]{Rebak}, one can show that for $0 < q < 1$, $r$ has three distinct positive roots.

\medskip

Now, in Theorems \ref{thm:e27}--\ref{thm:e81sqrt3}, we provide five examples for Theorem~\ref{thm:ninth}\eqref{9:phi-1-9}. 

\begin{theorem}\label{thm:e27} We have
\begin{multline*}
\frac{\varphi(e^{-27\pi})}{\varphi(e^{-3\pi})} = \frac{1}{3}\Bigg\{1 + \frac{(16(11\sqrt{3} - 19))^{1/9}}{\sqrt{3}}\\
\times
\Bigg(
\bigg(1 + \frac{\sqrt{3}}{2\cos\frac{4\pi}{9}}\bigg)^{1/3} + 
\bigg(1 + \frac{\sqrt{3}}{2\cos\frac{2\pi}{9}}\bigg)^{1/3} + 
\bigg(1 - \frac{\sqrt{3}}{2\cos\frac{\pi}{9}}\bigg)^{1/3}
\Bigg)\Bigg\}.
\end{multline*}
\end{theorem}

Note that Theorem~\ref{thm:e27} is the ninth degree analogue of Theorem~\ref{thm:enigmatic}.

\begin{proof}[First Proof of Theorem~\ref{thm:e27}] We use Theorem~\ref{thm:ninth} with $q = \exp(-\pi/3)$. First, by using Lem\-ma~\ref{lemma:transform} with $n = 729 = 27^2$, we rewrite Theorem~\ref{thm:ninth}\eqref{9:phi-1-9} as
\begin{equation}\label{eq:e27-rewrite}
\frac{\varphi(e^{-27\pi})}{\varphi(e^{-3\pi})} = \frac{1}{3\sqrt{3}}(1 + u_1 + u_2 + u_3 + u_4).
\end{equation}
From \cite[p.~189]{BerndtV}, $G_1 = 1$, and from \cite{RamanujanModularPi}, \cite[p.~24]{RamanujanCollected}, \cite[p.~189]{BerndtV}, with the use of Lemma~\ref{lemma:G},
\begin{equation}\label{eq:G9}
G_{1/9} = G_9 = \bigg(\frac{1 + \sqrt{3}}{\sqrt{2}}\bigg)^{1/3} = (2 + \sqrt{3})^{1/6}.
\end{equation}
By Lemma~\ref{lemma:p-u_3-G}\eqref{p-G} with $n = 1/9$, for Theorem~\ref{thm:ninth}\eqref{9:p}, we have
\begin{equation}\label{eq:e27p}
p = \frac{2\sqrt{2} G_{1/9}}{G_1 G_9^6} = \frac{2\sqrt{2}}{G_9^5} = \frac{4 \cdot 2^{1/3}}{(1 + \sqrt{3})^{5/3}} = (4(\sqrt{3} - 1)^5)^{1/3} = (16(11\sqrt{3} - 19))^{1/3}.
\end{equation}
By Theorem~\ref{thm:ninth}\eqref{9:u3} and Lemma~\ref{lemma:transform} with $n = 9$, or by Lemma~\ref{lemma:p-u_3-G}\eqref{u_3-G}, we find that
\begin{equation}\label{eq:e27u3}
u_3 = \frac{\varphi(q)}{\varphi(q^9)} - 1 = \frac{\varphi(e^{-\pi/3})}{\varphi(e^{-3\pi})} - 1 = \sqrt{3} - 1.
\end{equation}
Now, we have all the coefficients of $r$ given in Theorem~\ref{thm:ninth}\eqref{9:r}. We have to determine the zeros of
\begin{equation}\label{eq:e27r}
r(\xi) = \xi^3 - 6(\sqrt{3} - 1)\xi^2 + 12(\sqrt{3} - 1)^3\xi - 4(\sqrt{3} - 1)^5.
\end{equation}

By a proper polynomial transformation, we find that
\begin{equation}\label{eq:e27r-transform}
\frac{r(2(\sqrt{3} - 1)(1 - \sqrt{2}(\sqrt{3} - 1)\xi))}{8\sqrt{2}(\sqrt{3} - 1)^4(\sqrt{3} - 2)} = 4\xi^3 - 3\xi + \frac{\sqrt{2 - \sqrt{3}}}{2}.
\end{equation}
Set $\theta = 7\pi/36, 17\pi/36,$ and $31\pi/36$ in the power-reduction trigonometric identity \cite[p.~32]{GradshteynRyzhik7}, \cite[pp.~347--348]{BerndtIII}, \cite[p.~39]{BerndtIV} 
\begin{equation}\label{eq:power-reduction-cos}
4\cos^3 \theta =  3 \cos \theta + \cos 3\theta.
\end{equation}
By \eqref{eq:e27r}--\eqref{eq:power-reduction-cos}, with the values of $\cos(3 \theta) = -(2 - \sqrt{3})^{1/2}/2$, we find that the roots of $r$ are
\begin{equation*}
2(\sqrt{3} - 1)\bigg(1 - \sqrt{2}(\sqrt{3} - 1)\cos\bigg(\frac{k\pi}{36}\bigg)\bigg),\qquad k = 7, 17, 31.
\end{equation*}
In the case of $k = 7$, by angle sum and difference identities \cite[p.~29]{GradshteynRyzhik7}, observe that 
\begin{align}
\cos\bigg(\frac{7\pi}{36}\bigg) &= 
\cos\bigg(\frac{\pi}{12}\bigg)\cos\bigg(\frac{\pi}{9}\bigg) - 
\sin\bigg(\frac{\pi}{12}\bigg)\sin\bigg(\frac{\pi}{9}\bigg)\notag\\ 
&= \frac{\sqrt{2 + \sqrt{3}}}{2}\cos\bigg(\frac{\pi}{9}\bigg) - 
\frac{\sqrt{2 - \sqrt{3}}}{2}\sin\bigg(\frac{\pi}{9}\bigg)\notag\\
&=
\frac{1}{\sqrt{2}(\sqrt{3} - 1)}\cos\bigg(\frac{\pi}{9}\bigg) - 
\frac{2 - \sqrt{3}}{\sqrt{2}(\sqrt{3} - 1)}\sin\bigg(\frac{\pi}{9}\bigg).\label{eq:trig-transform}
\end{align}
Taking similar steps in the other two cases as well, we find that
\begin{align}
\alpha &= 2(\sqrt{3} - 1)\bigg(1 - \cos\bigg(\frac{\pi}{9}\bigg) + (2 - \sqrt{3})\sin\bigg(\frac{\pi}{9}\bigg)\bigg)\label{eq:e27-alpha},\\
\beta &= 2(\sqrt{3} - 1)\bigg(1 + (2 - \sqrt{3})\cos\bigg(\frac{\pi}{9}\bigg) - \sin\bigg(\frac{\pi}{9}\bigg)\bigg)\label{eq:e27-beta},\\
\gamma &= 2(\sqrt{3} - 1)\bigg(1 + (\sqrt{3} - 1)\bigg(\cos\bigg(\frac{\pi}{9}\bigg) + \sin\bigg(\frac{\pi}{9}\bigg)\bigg)\bigg)\label{eq:e27-gamma}.
\end{align}

By rearrangement and product-to-sum identities \cite[p.~29]{GradshteynRyzhik7}, with $\cos(5\pi/9) = -\cos(4\pi/9)$ and $\sin(5\pi/9) = \sin(4\pi/9),$ and with the values $\cos(\pi/3) = 1/2$ and $\sin(\pi/3) = \sqrt{3}/2$, we find that
\begin{align}
\frac{\beta}{\alpha} &= \frac{1 + (2-\sqrt{3})\cos(\frac{\pi}{9}) - \sin(\frac{\pi}{9})}{1 - \cos(\frac{\pi}{9}) + (2 - \sqrt{3})\cos(\frac{\pi}{9})}\notag\\
&= 1 + \sqrt{3}\frac{(\sqrt{3} - 1)(\cos(\frac{\pi}{9}) - \sin(\frac{\pi}{9}))}{1 - \cos(\frac{\pi}{9}) + (2 - \sqrt{3})\cos(\frac{\pi}{9})}\notag\\
&= 1 + \frac{\sqrt{3}}{2\cos(\frac{4\pi}{9})}\cdot\frac{(\sqrt{3} - 1)(\frac{1}{2} - \cos(\frac{4\pi}{9}) + \frac{\sqrt{3}}{2} - \sin(\frac{4\pi}{9}))}{1 - \cos(\frac{\pi}{9}) + (2 - \sqrt{3})\cos(\frac{\pi}{9})}\notag\\
&= 1 + \frac{\sqrt{3}}{2\cos(\frac{4\pi}{9})}\cdot\frac{1 - (\sqrt{3} - 1)(\sin(\frac{4\pi}{9}) + \cos(\frac{4\pi}{9}))}{1 - (\cos(\frac{\pi}{9}) - (2 - \sqrt{3})\cos(\frac{\pi}{9}))}\label{eq:trig-transform2}.
\end{align}
The last fraction in \eqref{eq:trig-transform2} is equal to 1. To see this, take similar steps as in \eqref{eq:trig-transform}, i.e.,
\begin{align*}
\frac{1}{\sqrt{2}}\sin\bigg(\frac{4\pi}{9}\bigg) + \frac{1}{\sqrt{2}}\cos\bigg(\frac{4\pi}{9}\bigg) = \frac{1}{\sqrt{2}(\sqrt{3} - 1)}\cos\bigg(\frac{\pi}{9}\bigg) - \frac{2 - \sqrt{3}}{\sqrt{2}(\sqrt{3} - 1)}\sin\bigg(\frac{\pi}{9}\bigg) = \cos\bigg(\frac{7\pi}{36}\bigg).
\end{align*}
Thus, with Theorem~\ref{thm:ninth}\eqref{9:u124} in mind, for the other two cases in a similar way, we obtain
\begin{align}
\frac{\beta}{\alpha} = 1 + \frac{\sqrt{3}}{2\cos\frac{4\pi}{9}}, \quad
\frac{\gamma}{\beta} = 1 + \frac{\sqrt{3}}{2\cos\frac{2\pi}{9}}, \text{\quad and \quad}
\frac{\alpha}{\gamma} = 1 - \frac{\sqrt{3}}{2\cos\frac{\pi}{9}}\label{eq:trig-transform3}.
\end{align}

Lastly, we show that our choice of the order $(\alpha, \beta, \gamma)$ is correct. By \eqref{eq:e27-alpha}--\eqref{eq:e27-gamma}, and by the inequalities $\sqrt{3} > 3/2$ and $\cos(\pi/9) > \sin(\pi/9)$, we have $\alpha - \beta < 0, \beta - \gamma < 0,$ and $\gamma - \alpha > 0$, therefore the condition Lemma~\ref{lemma:9:a,b,c-order}\eqref{9:a,b,c-order-cond1} holds. The condition Lemma~\ref{lemma:9:a,b,c-order}\eqref{9:a,b,c-order-cond2} is satisfied by using \eqref{eq:trig-transform3} with the inequalities $0 < \cos(4\pi/9) < \cos(2\pi/9) < \cos(\pi/9) < 1$. After recalling that $p$ and $u_3$ are obtained in \eqref{eq:e27p} and \eqref{eq:e27u3}, respectively, by Theorem~\ref{thm:ninth}\eqref{9:u124}, we conclude that
\begin{equation}\label{eq:e27-u124}
u_1 = \bigg\{p\bigg(1 + \frac{\sqrt{3}}{2\cos\tfrac{4\pi}{9}}\bigg)\bigg\}^{1/3}, \, 
u_2 = \bigg\{p\bigg(1 + \frac{\sqrt{3}}{2\cos\tfrac{2\pi}{9}}\bigg)\bigg\}^{1/3},\,\, \text{and} \,\,
u_4 = \bigg\{p\bigg(1 - \frac{\sqrt{3}}{2\cos\tfrac{\pi}{9}}\bigg)\bigg\}^{1/3}.
\end{equation}
We complete the proof by rearranging \eqref{eq:e27-rewrite}.
\end{proof}

Chapter 20 of Ramanujan's second notebook \cite[pp.~241--251]{RamanujanEarlierII}, \cite[pp.~325--453]{BerndtIII} contains modular equations of higher and composite degrees. In the first entry of this chapter \cite[p.~241]{RamanujanEarlierII}, among theta functions identities, there are two trigonometric identities. These are
\begin{equation}\label{eq:Ramanujan-first-trig}
\cos\bigg(\frac{2\pi}{9}\bigg) + \cos\bigg(\frac{4\pi}{9}\bigg) = \cos\bigg(\frac{\pi}{9}\bigg)
\end{equation}
and
\begin{equation}\label{eq:Ramanujan-second-trig}
\frac{1}{\cos\frac{4\pi}{9}} + \frac{1}{\cos\frac{2\pi}{9}} = \frac{1}{\cos\frac{\pi}{9}} + 6.
\end{equation}
Ramanujan gave them in degrees instead of radians. In Berndt's book \cite[pp.~345--349]{BerndtIII}, these are the third and forth identities of Entry~1(iii), and are proved by the identity \eqref{eq:power-reduction-cos}. Berndt mentions that this proof does not relate these identities to the rest of the material in Entry~1, but Ramanujan clearly had a connection in mind.

From the first proof of Theorem~\ref{thm:e27}, we see a connection with theta functions. By using Lemma~\ref{lemma:u_1,u_2,u_4-split}\eqref{split-iii} with $q = \exp(-\pi/3)$, with the values of $p$ and $u_3$ given in \eqref{eq:e27p} and \eqref{eq:e27u3}, respectively, and with the obtained values for $u_1, u_2,$ and $u_4$ in \eqref{eq:e27-u124}, we find \eqref{eq:Ramanujan-second-trig}. Similarly, by trigonometric transformations, we find that \eqref{eq:e27-u124} can be written in the form
\begin{align*}
u_1 &= \bigg\{p\bigg(1 + \sqrt{3} + 2\sqrt{3}\cos\bigg(\frac{\pi}{9}\bigg)\bigg)\bigg\}^{1/3},\\
u_2 &= \bigg\{p\bigg(1+ \sqrt{3} - 2\sqrt{3}\cos\bigg(\frac{4\pi}{9}\bigg)\bigg)\bigg\}^{1/3},\\
u_4 &= \bigg\{p\bigg(1 + \sqrt{3} - 2\sqrt{3}\cos\bigg(\frac{2\pi}{9}\bigg)\bigg)\bigg\}^{1/3}.
\end{align*}
Therefore, using these forms, we can deduce \eqref{eq:Ramanujan-first-trig} in the same way.

\medskip

Our second proof relies on the entry of Ramanujan from his second notebook, discussed in Section~\ref{section:cubic}. In the following, the notations $\alpha, \beta,\gamma,$ and $r$ have different meanings than those introduced in Theorem~\ref{thm:ninth} and used elsewhere.

\begin{proof}[Second Proof of Theorem~\ref{thm:e27}] By Lemma~\ref{lemma:9n}, for $n = 9$, we find that
\begin{equation}\label{eq:phi27}
\frac{\varphi(e^{-27\pi})}{\varphi(e^{-3\pi})} =  \frac{1}{3}\bigg(1 + \frac{\sqrt{2}G_{81}}{G_9^3}\bigg),
\end{equation}
with the value $G_9$ from \eqref{eq:G9}, and by Lemma~\ref{lemma:G81n}, for $n = 1$, with $G_1 = 1$ \cite[p.~189]{BerndtV},
\begin{equation}\label{eq:G81-expr}
G_{81}^3 = G_9 \frac{\sqrt{2}G_9 + 1}{\sqrt{2} - G_9}.
\end{equation}
Note by \eqref{eq:G9} that
\begin{equation}\label{eq:G9-notes}
(16(11\sqrt{3} - 19))^{1/9} = \sqrt{2}G_9^{-5/3} \text{\qquad and \qquad} \sqrt{3} - 1 = \sqrt{2}G_9^{-3}.
\end{equation}
Let $\alpha, \beta,$ and $\gamma$ denote the expressions under the cube roots in the statement, i.e.,
\begin{equation}\label{eq:e27-alpha,beta,gamma}
\alpha := 1 + \frac{\sqrt{3}}{2\cos\frac{2\pi}{9}}, \quad
\beta := 1 + \frac{\sqrt{3}}{2\cos\frac{4\pi}{9}}, \quad \text{and} \quad
\gamma := 1 - \frac{\sqrt{3}}{2\cos\frac{\pi}{9}}.
\end{equation}
Then, by using \eqref{eq:phi27}, \eqref{eq:G9-notes}, and \eqref{eq:e27-alpha,beta,gamma}, we can restate the theorem as
\begin{equation}\label{eq:alpha-beta-gamma-expr1}
\frac{\sqrt{3} G_{81}}{G_9^{4/3}} = \alpha^{1/3} + \beta^{1/3} + \gamma^{1/3}.
\end{equation}

By setting $\theta = 2\pi/9, 4\pi/9,$ and $8\pi/9$ in the power-reduction formula \eqref{eq:power-reduction-cos}, we find that $2\cos(2\pi/9), 2\cos(4\pi/9),$ and $2\cos(8\pi/9) = -2\cos(\pi/9)$ are the roots of
\begin{equation*}
r(x) := x^3 - 3x + 1.
\end{equation*}
Thus, by a proper polynomial transformation, $\alpha, \beta,$ and $\gamma$ are the roots of the polynomial
\begin{equation}\label{eq:e27-r-tilde}
\widetilde{r}(x) := (x-1)^3 r\bigg(\frac{\sqrt{3}}{x - 1}\bigg) = x^3 - 3(1 + \sqrt{3})x^2 + 3(1 + 2\sqrt{3})x - 1.
\end{equation}

By Entry~\ref{entry:Ramanujan356}, we know that if $\alpha, \beta,$ and $\gamma$ are the roots of
\begin{equation*}
x^3 - ax^2 + bx - 1 = 0,
\end{equation*}
then
\begin{equation*}
\alpha^{1/3} + \beta^{1/3} + \gamma^{1/3} = \big(a + 6 + 3\{A + B\}^{1/3} + 3\{A - B\}^{1/3}\big)^{1/3},
\end{equation*}
with
\begin{equation*}
A = \frac{ab + 9}{2} + 3a + 3b \text{\qquad and \qquad} B = \bigg(\frac{(ab + 9)^2}{4} - a^3 - b^3 - 27\bigg)^{1/2}.
\end{equation*}
We apply this theorem with $a = 3(1 + \sqrt{3})$ and $b = 3(1 + 2\sqrt{3})$. Since
\begin{equation*}
a + 6 = 3\sqrt{3}(1 + \sqrt{3}), \quad A + B = 54(1 + \sqrt{3}), \text{\quad and \quad} A - B = \tfrac{27}{2}(1 + \sqrt{3})^2,
\end{equation*}
and knowing that $\alpha, \beta,$ and $\gamma$ in \eqref{eq:e27-alpha,beta,gamma} are the roots of $\widetilde{r}$ in \eqref{eq:e27-r-tilde}, with \eqref{eq:G9} in mind, we find that
\begin{align}
\alpha^{1/3} + \beta^{1/3} + \gamma^{1/3} &= \bigg(3\sqrt{3}(1 + \sqrt{3}) + 9\sqrt[3]{2}(1 + \sqrt{3})^{1/3} + \frac{9}{\sqrt[3]{2}}(1 + \sqrt{3})^{2/3}\bigg)^{1/3}\notag\\
&= (3\sqrt{6}G_9^3 + 9\sqrt{2}G_9 + 9G_9^2)^{1/3}\label{eq:alpha-beta-gamma-expr2}.
\end{align}

Now, combining \eqref{eq:alpha-beta-gamma-expr1} and \eqref{eq:alpha-beta-gamma-expr2}, we find that we are required to prove that
\begin{equation}\label{eq:alpha-beta-gamma-comb}
\frac{3\sqrt{3} G_{81}^3}{G_9^{4}} = 3\sqrt{6}G_9^3 + 9\sqrt{2}G_9 + 9G_9^2.
\end{equation}
Substituting \eqref{eq:G81-expr} into \eqref{eq:alpha-beta-gamma-comb} yields
\begin{equation*}
\frac{\sqrt{2}G_9 + 1}{\sqrt{2} - G_9} = \sqrt{2}G_9^6 + \sqrt{6}G_9^4 + \sqrt{3}G_9^5.
\end{equation*}
After rearrangement, we may recast \eqref{eq:alpha-beta-gamma-comb} in the form
\begin{equation}\label{eq:e27-G9-eq}
\sqrt{2}G_9^7 - (2 - \sqrt{3})G_9^6 - 2\sqrt{3}G_9^4 + \sqrt{2}G_9 + 1 = 0.
\end{equation}
By factoring \eqref{eq:e27-G9-eq}, we see that it remains to show that
\begin{equation}\label{eq:e27-G9-factorization}
\frac{1}{4\sqrt{2}}\big(2G_9^3 - \sqrt{2}G_9^2 + (\sqrt{3} - 1)G_9 - \sqrt{2}\big)\big(2G_9^3 - \sqrt{6} - \sqrt{2}\big)\big(2G_9 + \sqrt{6} - \sqrt{2}\big) = 0.
\end{equation}
From \eqref{eq:G9}, we know that $2G_9^3 - \sqrt{6} - \sqrt{2} = 0$, and thus \eqref{eq:e27-G9-factorization} holds to complete the proof.
\end{proof}

\begin{theorem}\label{thm:e9sqrt3} We have
\begin{multline*}
\frac{\varphi(e^{-9\pi\sqrt{3}})}{\varphi(e^{-3\pi\sqrt{3}})} = 3^{-5/4}\\
\times \Bigg(
1 + 
\bigg(p\frac{1 - a\sin\frac{\pi}{9}}{1 - a\sin\frac{2\pi}{9}}\bigg)^{1/3} + 
\bigg(p\frac{1 + a\sin\frac{4\pi}{9}}{1 - a\sin\frac{\pi}{9}}\bigg)^{1/3} + (2^{2/3} - 2^{1/3}) + 
\bigg(p\frac{1 - a\sin\frac{2\pi}{9}}{1 + a\sin\frac{4\pi}{9}}\bigg)^{1/3}
\Bigg),
\end{multline*}
where
\begin{align*}
p &= 2(2^{1/3} - 1)^2,\\
a &= 2(2^{2/3} - 1)^{1/2}.
\end{align*}
\end{theorem}

\begin{proof}
We apply Theorem~\ref{thm:ninth} with $q = \exp(-\pi/\sqrt{3})$. By using Lemma~\ref{lemma:transform} with $n = (9\sqrt{3})^2$, we rewrite Theorem~\ref{thm:ninth}\eqref{9:phi-1-9} as
\begin{equation}\label{eq:e9sqrt3-rewrite}
\frac{\varphi(e^{-9\pi\sqrt{3}})}{\varphi(e^{-3\pi\sqrt{3}})} = 3^{-5/4}(1 + u_1 + u_2 + u_3 + u_4).
\end{equation}
From~\cite[pp.~189--190]{BerndtV}, by Lemma~\ref{lemma:G},
\begin{equation}\label{eq:G3G27}
G_{1/3} = G_3 = 2^{1/12} \qquad \text{and} \qquad G_{27} = 2^{1/12}(2^{1/3} - 1)^{-1/3}.
\end{equation}
By Lemma~\ref{lemma:p-u_3-G}\eqref{p-G},\eqref{u_3-G} with $n = 1/3$, for Theorem~\ref{thm:ninth}\eqref{9:p},\eqref{9:u3}, we find
\begin{equation}\label{eq:e9sqrt3-p-u3}
p = \frac{2\sqrt{2}}{G_{27}^6} = 2(2^{1/3} - 1)^2 \qquad \text{and} \qquad u_3 = \frac{\sqrt{2} G_3}{G_{27}^3} = 2^{1/3}(2^{1/3} - 1).
\end{equation}
We have to determine the zeros of $r$ in Theorem~\ref{thm:ninth}\eqref{9:r}, i.e., the zeros of
\begin{equation}\label{eq:e9sqrt3r}
r(\xi) = \xi^3 - 6(2^{1/3} - 1)\xi^2 + 12(2^{1/3} - 1)^2(2 - 2^{2/3})\xi - 8(2^{1/3} - 1)^6 = 0.
\end{equation}

By a polynomial transformation, we find that
\begin{equation}\label{eq:e9sqrt3r-transform}
-\frac{\sqrt{2}}{a^3 p^{3/2}}r((2p)^{1/2}(1 - a\xi)) = 4\xi^3 - 3\xi + \frac{\sqrt{3}}{2}.
\end{equation}
Set $\theta = \pi/9, 2\pi/9,$ and $-4\pi/9$ in the power-reduction trigonometric identity \cite[p.~32]{GradshteynRyzhik7}
\begin{equation}\label{eq:power-reduction-sin}
4\sin^3 \theta = 3\sin \theta - \sin 3\theta.
\end{equation}
By \eqref{eq:e9sqrt3r}--\eqref{eq:power-reduction-sin}, with the values of $\sin(3 \theta) = \sqrt{3}/2$, since sine is odd, the roots of $r$ are
\begin{equation*}
\alpha = (2p)^{1/2}\bigg(1 - a\sin\bigg(\frac{2\pi}{9}\bigg)\bigg), \,\,\,
\beta = (2p)^{1/2}\bigg(1 - a\sin\bigg(\frac{\pi}{9}\bigg)\bigg), \,\,\,
\gamma = (2p)^{1/2}\bigg(1 + a\sin\bigg(\frac{4\pi}{9}\bigg)\bigg).
\end{equation*}

\emph{Mathematica} confirms that both conditions in Lemma~\ref{lemma:9:a,b,c-order}\eqref{9:a,b,c-order-cond1},\eqref{9:a,b,c-order-cond2} hold. This can be done directly, or by excluding by numerical evaluations all possible orders of $\alpha, \beta,$ and $\gamma$, which do not meet the conditions. Thus $(\alpha, \beta, \gamma)$ is the correct order of the roots. With \eqref{eq:e9sqrt3-rewrite} in mind, we recall that $p$ and $u_3$ are given in \eqref{eq:e9sqrt3-p-u3}, and by Theorem~\ref{thm:ninth}\eqref{9:u124}, we find that
\begin{equation*}
u_1 = \bigg(p\frac{1 - a\sin\frac{\pi}{9}}{1 - a\sin\frac{2\pi}{9}}\bigg)^{1/3},\quad
u_2 = \bigg(p\frac{1 + a\sin\frac{4\pi}{9}}{1 - a\sin\frac{\pi}{9}}\bigg)^{1/3},\quad \text{and} \quad
u_4 = \bigg(p\frac{1 - a\sin\frac{2\pi}{9}}{1 + a\sin\frac{4\pi}{9}}\bigg)^{1/3},
\end{equation*}
which completes the proof.
\end{proof}

\begin{theorem}\label{thm:e27sqrt3} We have
\begin{equation*}
\frac{\varphi(e^{-27\pi\sqrt{3}})}{\varphi(e^{-\pi\sqrt{3}})} = 3^{-7/4}
\Bigg(
1 + 
\bigg(p\frac{1 + a\cos\frac{4\pi}{9}}{1 - a\cos\frac{\pi}{9}}\bigg)^{1/3} + 
\bigg(p\frac{1 + a\cos\frac{2\pi}{9}}{1 + a\cos\frac{4\pi}{9}}\bigg)^{1/3} + 2^{1/3} + 
\bigg(p\frac{1 - a\cos\frac{\pi}{9}}{1 + a\cos\frac{2\pi}{9}}\bigg)^{1/3}
\Bigg),
\end{equation*}
where
\begin{align*}
p &= 2(1 + 2^{1/3} + 2^{2/3})^{1/3},\\
a &= 16p^{-3} = 2(2^{1/3} - 1).
\end{align*}
\end{theorem}

\begin{proof}
We use Theorem~\ref{thm:ninth} with $q = \exp(-\pi/\sqrt{27})$. By applying Lemma~\ref{lemma:transform} for $n =(27\sqrt{3})^2$,  we transform Theorem~\ref{thm:ninth}\eqref{9:phi-1-9} into
\begin{equation}\label{eq:e27sqrt3-rewrite}
\frac{\varphi(e^{-27\pi\sqrt{3}})}{\varphi(e^{-\pi\sqrt{3}})} = 3^{-7/4}(1 + u_1 + u_2 + u_3 + u_4).
\end{equation}
Utilize the values of $G_3$ and $G_{27}$ in \eqref{eq:G3G27}, with Lemma~\ref{lemma:G}. By Lemma~\ref{lemma:p-u_3-G}\eqref{p-G},\eqref{u_3-G} with $n = 1/27$, for Theorem~\ref{thm:ninth}\eqref{9:p},\eqref{9:u3}, we obtain
\begin{equation}\label{eq:e27sqrt3-p-u3}
p = \frac{2\sqrt{2} G_{27}}{G_3^7} = \frac{2}{(2^{1/3} - 1)^{1/3}} = 2(1 + 2^{1/3} + 2^{2/3})^{1/3} \qquad \text{and} \qquad u_3 = \frac{\sqrt{2}}{G_3^2} = 2^{1/3}.
\end{equation}
We have to determine the roots of $r$ in Theorem~\ref{thm:ninth}\eqref{9:r}, i.e., the roots of
\begin{equation}\label{eq:e27sqrt3r}
r(\xi) = \xi^3 - 2(1 + 2^{1/3})^2\xi^2 + 4(1 + 2^{1/3})(2 + 2^{2/3})\xi - 8(1 + 2^{1/3} + 2^{2/3}) = 0.
\end{equation}

By a polynomial transformation, we arrive at
\begin{equation}\label{eq:e27sqrt3r-transform}
\frac{3}{8a}r\bigg(\frac{2}{3}(1 + 2^{1/3})^2(1 + a \xi)\bigg) = 4\xi^3 - 3\xi + \frac{1}{2}.
\end{equation}
By \eqref{eq:e27sqrt3r}, \eqref{eq:e27sqrt3r-transform}, and by setting $\theta = 2\pi/9, 4\pi/9,$ and $8\pi/9$ in the power-reduction identity \eqref{eq:power-reduction-cos}, with the values of $\cos(3 \theta) = -1/2$, we find that the roots of $r$ are
\begin{equation*}
\alpha = b\bigg(1 - a\cos\bigg(\frac{\pi}{9}\bigg)\bigg), \quad
\beta = b\bigg(1 + a\cos\bigg(\frac{4\pi}{9}\bigg)\bigg), \quad
\gamma = b\bigg(1 + a\cos\bigg(\frac{2\pi}{9}\bigg)\bigg),
\end{equation*}
where $b = 2(1 + 2^{1/3})^2/3$, and we used the fact that $\cos(8\pi/9) = -\cos(\pi/9)$.

As in the end of the proof of Theorem~\ref{thm:e9sqrt3}, we verify with \emph{Mathematica} that the conditions in Lemma~\ref{lemma:9:a,b,c-order}\eqref{9:a,b,c-order-cond1},\eqref{9:a,b,c-order-cond2} hold, thus $(\alpha, \beta, \gamma)$ is the correct order of the roots. In view of \eqref{eq:e27sqrt3-rewrite}, we recall that $p$ and $u_3$ are given in \eqref{eq:e27sqrt3-p-u3}, and by Theorem~\ref{thm:ninth}\eqref{9:u124}, we conclude that
\begin{equation*}
u_1 = \bigg(p\frac{1 + a\cos\frac{4\pi}{9}}{1 - a\cos\frac{\pi}{9}}\bigg)^{1/3},\quad
u_2 = \bigg(p\frac{1 + a\cos\frac{2\pi}{9}}{1 + a\cos\frac{4\pi}{9}}\bigg)^{1/3},\quad \text{and} \quad
u_4 = \bigg(p\frac{1 - a\cos\frac{\pi}{9}}{1 + a\cos\frac{2\pi}{9}}\bigg)^{1/3}.\tag*{\qedhere}
\end{equation*}
\end{proof}

By combining the results of Theorems~\ref{thm:e9sqrt3} and~\ref{thm:e27sqrt3} with Lemma~\ref{lemma:u_1,u_2,u_4-split}\eqref{split-iii}, after simplification, we obtain the trigonometric identities
\begin{align*}
\frac{1 - a\sin\frac{\pi}{9}}{1 - a\sin\frac{2\pi}{9}} + \frac{1 + a\sin\frac{4\pi}{9}}{1 - a\sin\frac{\pi}{9}} + \frac{1 - a\sin\frac{2\pi}{9}}{1 + a\sin\frac{4\pi}{9}} = (1 + 2^{1/3})^3 (2 + 2^{1/3})&, \qquad a = 2(2^{2/3} - 1)^{1/2},\\
\intertext{and}
\frac{1 + a\cos\frac{4\pi}{9}}{1 - a\cos\frac{\pi}{9}} + \frac{1 + a\cos\frac{2\pi}{9}}{1 + a\cos\frac{4\pi}{9}} + \frac{1 - a\cos\frac{\pi}{9}}{1 + a\cos\frac{2\pi}{9}} = 3 \cdot 2^{1/3}&, \qquad a = 2(2^{1/3} - 1).
\end{align*}
On the one hand, these identities are analogous to Ramanujan's identities \eqref{eq:Ramanujan-second-trig} and \eqref{eq:Ramanujan-first-trig}, and on the other hand, to an identity for the septic case given in \cite[Lemma~3.8]{Rebak}, \cite[Example~3.3]{BerndtKimZaharescu}. 

Another family of related identities can be obtained by Corollary~\ref{cor:u_1,u_2,u_4-reciprocal}. With $q = \exp(-\pi/3)$, with the values of $p$ and $u_3$ from \eqref{eq:e27p} and \eqref{eq:e27u3}, respectively, and with the values of $u_1, u_2,$ and $u_4$ in \eqref{eq:e27-u124}, after rearrangement, we find that
\begin{equation*}
\frac{1}{2\cos(\frac{2\pi}{9}) + \sqrt{3}} + \frac{1}{2\cos(\frac{4\pi}{9}) + \sqrt{3}} = \frac{1}{2\cos(\frac{\pi}{9}) - \sqrt{3}} - 6,
\end{equation*}
which is similar to Ramanujan's identity \eqref{eq:Ramanujan-second-trig}. By combining the findings of Theorems~\ref{thm:e9sqrt3} and~\ref{thm:e27sqrt3} with Corollary~\ref{cor:u_1,u_2,u_4-reciprocal}, and simplifying, we arrive at
\begin{align*}
\frac{1 - a\sin\frac{2\pi}{9}}{1 - a\sin\frac{\pi}{9}} + \frac{1 - a\sin\frac{\pi}{9}}{1 + a\sin\frac{4\pi}{9}} + \frac{1 + a\sin\frac{4\pi}{9}}{1 - a\sin\frac{2\pi}{9}} = \frac{3}{(2^{1/3} - 1)^3}&, \qquad a = 2(2^{2/3} - 1)^{1/2},\\
\intertext{and}
\frac{1 - a\cos\frac{\pi}{9}}{1 + a\cos\frac{4\pi}{9}} + \frac{1 + a\cos\frac{4\pi}{9}}{1 + a\cos\frac{2\pi}{9}} + \frac{1 + a\cos\frac{2\pi}{9}}{1 - a\cos\frac{\pi}{9}} = \frac{9}{2^{1/3} + 1}&, \qquad a = 2(2^{1/3} - 1).
\end{align*}

Lastly, a third family of trigonometric identities can be found by comparing the values of Theorems~\ref{thm:e27}, \ref{thm:e9sqrt3}, and \ref{thm:e27sqrt3} with the values in Theorem~\ref{thm:alternative}\eqref{eq:e27-alt},\eqref{eq:e9sqrt3-alt}, and \eqref{eq:e27sqrt3-alt}, respectively. After simplification, we obtain the identities
\begin{multline*}
\bigg(1 + \frac{\sqrt{3}}{2\cos\frac{4\pi}{9}}\bigg)^{1/3} + 
\bigg(1 + \frac{\sqrt{3}}{2\cos\frac{2\pi}{9}}\bigg)^{1/3} + 
\bigg(1 - \frac{\sqrt{3}}{2\cos\frac{\pi}{9}}\bigg)^{1/3}\\
 =\sqrt{3}(7 - 4\sqrt{3})^{1/9}\left(\frac{(2(\sqrt{3} + 1))^{1/3} + 1}{(2(\sqrt{3} - 1))^{1/3} - 1}\right)^{1/3},
\end{multline*}
\begin{multline*}
\bigg(\frac{1 - a\sin\frac{\pi}{9}}{1 - a\sin\frac{2\pi}{9}}\bigg)^{1/3} + 
\bigg(\frac{1 + a\sin\frac{4\pi}{9}}{1 - a\sin\frac{\pi}{9}}\bigg)^{1/3} + 
\bigg(\frac{1 - a\sin\frac{2\pi}{9}}{1 + a\sin\frac{4\pi}{9}}\bigg)^{1/3}
 = (2^{1/3} + 1)^2, \quad a = 2(2^{2/3} - 1)^{1/2},
\end{multline*}
and
\begin{multline*}
\bigg(\frac{1 + a\cos\frac{4\pi}{9}}{1 - a\cos\frac{\pi}{9}}\bigg)^{1/3} + 
\bigg(\frac{1 + a\cos\frac{2\pi}{9}}{1 + a\cos\frac{4\pi}{9}}\bigg)^{1/3} + 
\bigg(\frac{1 - a\cos\frac{\pi}{9}}{1 + a\cos\frac{2\pi}{9}}\bigg)^{1/3}\\
= \frac{3^{1/3}(2^{1/3} - 1)^{4/9}(2^{1/3} + (2^{1/3} - 1)^{1/3})}{(3 - 2 \cdot 3^{1/3})^{1/3}}, \quad a = 2(2^{1/3} - 1),
\end{multline*}
respectively. Further investigation of these identities may be of interest.

\medskip
The forms of our last two examples are more complicated.

\begin{theorem}\label{thm:e81} We have
\begin{equation}\label{eq:e81}
\frac{\varphi(e^{-81\pi})}{\varphi(e^{-\pi})} = \frac{1}{9}(1 + u_1 + u_2 + u_3 + u_4),
\end{equation}
where
\begin{equation*}
u_1 = \bigg(\frac{\beta p}{\alpha}\bigg)^{1/3}, \quad
u_2 = \bigg(\frac{\gamma p}{\beta}\bigg)^{1/3}, \quad
u_3 = (2(\sqrt{3} + 1))^{1/3}, \quad
u_4 = \bigg(\frac{\alpha p}{\gamma}\bigg)^{1/3},
\end{equation*}
with
\begin{align*}
\alpha &= \frac{2}{3}\bigg(a - b \cos\bigg(\frac{\pi}{9}\bigg) - c \cos \bigg(\frac{2\pi}{9}\bigg)\bigg),\\
\beta &= \frac{2}{3}\bigg(a + d \cos\bigg(\frac{\pi}{9}\bigg) - b \cos \bigg(\frac{2\pi}{9}\bigg)\bigg),\\
\gamma &= \frac{2}{3}\bigg(a - c \cos\bigg(\frac{\pi}{9}\bigg) + d \cos \bigg(\frac{2\pi}{9}\bigg)\bigg),
\end{align*}
and
\begin{align*}
p &= 2\left(2(\sqrt{3} -1)\left(\frac{(2(\sqrt{3} + 1))^{1/3} + 1}{(2(\sqrt{3} - 1))^{1/3} - 1}\right)\right)^{1/3},\\
a &= 2(2(\sqrt{3} + 1))^{1/3} + (2(\sqrt{3} + 1))^{2/3} + \sqrt{3} + 1,\\
b &= 2\big((2(\sqrt{3} + 1))^{1/3} - (2(\sqrt{3} - 1))^{1/3}\big),\\
c &= 2\big((2(\sqrt{3} - 1))^{1/3} - (2(\sqrt{3} + 1))^{1/3} - \sqrt{3} - 1\big),\\
d &= 2(\sqrt{3} - 1).
\end{align*}
\end{theorem}

\begin{proof}
We appeal to Theorem~\ref{thm:ninth} with $q = \exp(-\pi/9)$. By using Lemma~\ref{lemma:transform} with $n = 81^2$, we rewrite Theorem~\ref{thm:ninth}\eqref{9:phi-1-9} as in \eqref{eq:e81}. Recall the value $G_1 = 1$ and the value of $G_9$ from \eqref{eq:G9}. The value of $G_{81}$ can be obtained by \eqref{eq:G81-expr}, or from \cite[p.~193]{BerndtV}, by Lemma~\ref{lemma:G},
\begin{equation*}
G_{1/81} = G_{81} = \left(\frac{(2(\sqrt{3} + 1))^{1/3} + 1}{(2(\sqrt{3} - 1))^{1/3} - 1}\right)^{1/3}.
\end{equation*}
By Lemma~\ref{lemma:p-u_3-G}\eqref{p-G},\eqref{u_3-G} with $n = 1/3$, and by Lemma~\ref{lemma:G}, for Theorem~\ref{thm:ninth}\eqref{9:p},\eqref{9:u3}, we find
\begin{equation*}
p = \frac{2\sqrt{2}G_{81}}{G_9} \qquad \text{and} \qquad u_3 = \sqrt{2} G_9.
\end{equation*}

We verify with \emph{Mathematica} that $\alpha, \beta,$ and $\gamma$ are the roots of $r$ in Theorem~\ref{thm:ninth}\eqref{9:r}, and that the conditions in Lemma~\ref{lemma:9:a,b,c-order}\eqref{9:a,b,c-order-cond1},\eqref{9:a,b,c-order-cond2} hold; therefore $(\alpha, \beta, \gamma)$ is the correct order of the roots, and by Theorem~\ref{thm:ninth}\eqref{9:u124}, the proof is complete.
\end{proof}

\begin{theorem}\label{thm:e81sqrt3} We have
\begin{equation}\label{eq:e81sqrt3}
\frac{\varphi(e^{-81\pi\sqrt{3}})}{\varphi(e^{-\pi\sqrt{3}})} = \frac{1}{9}(1 + u_1 + u_2 + u_3 + u_4),
\end{equation}
where
\begin{equation*}
u_1 = \bigg(\frac{\beta p}{\alpha}\bigg)^{1/3}, \quad
u_2 = \bigg(\frac{\gamma p}{\beta}\bigg)^{1/3}, \quad
u_3 = (2(1 + 2^{1/3} + 2^{2/3}))^{1/3}, \quad
u_4 = \bigg(\frac{\alpha p}{\gamma}\bigg)^{1/3},
\end{equation*}
with
\begin{align*}
\alpha &= \frac{1}{9}\bigg(a - b \cos\bigg(\frac{\pi}{9}\bigg) - c \cos \bigg(\frac{2\pi}{9}\bigg)\bigg),\\
\beta &= \frac{1}{9}\bigg(a + d \cos\bigg(\frac{\pi}{9}\bigg) - b \cos \bigg(\frac{2\pi}{9}\bigg)\bigg),\\
\gamma &= \frac{1}{9}\bigg(a - c \cos\bigg(\frac{\pi}{9}\bigg) + d \cos \bigg(\frac{2\pi}{9}\bigg)\bigg),
\end{align*}
and
\begin{align*}
p &= \frac{2(2^{1/3} + 2^{2/3} + 3^{1/3})(2^{1/3} - 1)^{1/3}}{(9 - 2 \cdot 3^{4/3})^{1/3}},\\
a &= 4 \cdot 2^{1/3}3^{2/3}(1 + 2^{1/3}) + 2 \cdot 3^{1/3}(2^{1/3} + 2^{2/3})^2 + 6(1 + 2^{1/3} + 2^{2/3}),\\
b &= 4\big(1 + 3^{1/3}(3^{1/3} + 1) - 2 \cdot 2^{1/3}(3^{1/3} + 2^{1/3} - 1)\big),\\
c &= 4\big(1 - 2^{1/3}(3^{1/3} - 1)(3^{1/3} + 2^{1/3} - 1)\big),\\
d &= 4\big((2 - 2^{1/3})(1 + 2^{1/3}) - 3^{1/3}(2^{1/3} - 1)(1 + 2^{1/3} + 3^{1/3})\big).
\end{align*}
\end{theorem}

\begin{proof}
We invoke Theorem~\ref{thm:ninth} with $q = \exp(-\pi/\sqrt{243})$. Applying Lemma~\ref{lemma:transform} twice with $n = (81\sqrt{3})^2$ and $n = 3$ for Theorem~\ref{thm:ninth}\eqref{9:phi-1-9}, we find the representation in \eqref{eq:e81sqrt3}. With Lemma~\ref{lemma:G} in mind, we utilize the values of $G_3$ and $G_{27}$ from \eqref{eq:G3G27}. From Watson's paper \cite[pp.~100--105]{Watson3}, with Lemma~\ref{lemma:G}, we also know that
\begin{equation*}
G_{1/243} = G_{243} = \frac{2^{1/12}(2^{1/3} + 2^{2/3} + 3^{1/3})}{(9 - 2 \cdot 3^{4/3})^{1/3}}.
\end{equation*}
By Lemma~\ref{lemma:p-u_3-G}\eqref{p-G},\eqref{u_3-G} with $n = 1/243$, for Theorem~\ref{thm:ninth}\eqref{9:p},\eqref{9:u3}, we obtain the values
\begin{equation*}
p = \frac{2\sqrt{2}G_{243}}{G_{27} G_3^6} \qquad \text{and} \qquad u_3 = \frac{\sqrt{2} G_{27}}{G_3^3}.
\end{equation*}

\emph{Mathematica} confirms that $\alpha, \beta,$ and $\gamma$ are the roots of $r$ in Theorem~\ref{thm:ninth}\eqref{9:r}, and that the conditions in Lemma~\ref{lemma:9:a,b,c-order}\eqref{9:a,b,c-order-cond1},\eqref{9:a,b,c-order-cond2} hold. Thus, by Theorem~\ref{thm:ninth}\eqref{9:u124}, we finish the proof.
\end{proof}

From \cite{BerndtRebak2}, we know the value of $G_{729}$, and this allows us to evaluate $\varphi(e^{-243\pi})$, but the solutions of the corresponding cubic equation seem much more complicated.

In the articles of Son~\cite{Son} and the second author~\cite{Rebak}, we examined Ramanujan's septic theta function identitiy. Analogous cubic and quintic identities are the topic of the paper by Berndt and the second author~\cite{BerndtRebak2}. In this paper, we obtained a ninth degree theta function identity, and calculated five examples of this kind. We believe that these results are special cases of a grand theory that Ramanujan envisioned at the end of Section~12 of Chapter~20 in his second notebook \cite[p.~247]{RamanujanEarlierII}, \cite[p.~400]{BerndtIII}.

\subsubsection*{Acknowledgments} The comments of Professor Bruce C. Berndt are highly appreciated. The first author's research was supported by the National Research Foundation of Korea (NRF) grant funded by the Korea government (MSIT) (No. RS-2023-00240168).

\medskip

\end{document}